\documentclass[reqno,11pt]{amsart}
\usepackage{amsmath,amsfonts,amsthm,amssymb,enumerate,graphicx,hyperref,esint}
\usepackage[margin=1.1in]{geometry}
\usepackage{tikz,tikz-cd,mathtools,soul}

\let\div\relax
\DeclareMathOperator\div{div}

\DeclareMathOperator\curl{curl}
\DeclareMathOperator\Id{Id}

\DeclareMathOperator\tr{tr}
\DeclareMathOperator\dist{dist}
\newcommand\dd[0]{\partial}
\newcommand\grad[0]{\nabla}
\newcommand\ulin[0]{u^\flat}
\newcommand\unlin[0]{u^\sharp}
\newcommand\unlinI[0]{u^{\sharp,1}}
\newcommand\unlinII[0]{u^{\sharp,2}}
\newcommand\wlin[0]{\omega^\flat}
\newcommand\wnlin[0]{\omega^\sharp}
\newcommand\eq[1]{\begin{align}{#1}\end{align}}
\newcommand\eqn[1]{\begin{align*}{#1}\end{align*}}
\newcommand\Rd[0]{{\mathbb R^d}}

\newtheorem{thm}{Theorem}[section]
\newtheorem{lem}[thm]{Lemma}
\newtheorem{prop}[thm]{Proposition}
\theoremstyle{remark}

\newtheorem{remark}[thm]{Remark}

\title[High-dimensional Navier-Stokes equations]{A minimum critical blowup rate for the high-dimensional Navier-Stokes equations}
\author{Stan Palasek}
\thanks{The author is grateful to Terence Tao for many helpful discussions and comments on the manuscript. We also thank the anonymous referees for their careful reading and valuable suggestions. This work was partially funded by NSF grant DMS-1764034.}
\address{UCLA Department of Mathematics\\Los Angeles, CA, USA}
\email{palasek@math.ucla.edu}

\begin{document}

\begin{abstract}
    We prove quantitative regularity and blowup theorems for the incompressible Navier-Stokes equations in $\Rd$, $d\geq4$ when the solution lies in the critical space $L_t^\infty L_x^d$. Explicit subcritical bounds on the solution are obtained in terms of the critical norm. A consequence is that $\|u(t)\|_{L_x^d(\Rd)}$ grows at a minimum rate of $(\log\log\log\log(T_*-t)^{-1})^c$ along a sequence of times approaching a hypothetical blowup at $T_*$. We use a quantitative framework inspired by Tao \cite{tao}, with some new elements to deal with the lack of Leray's epochs of regularity in the high-dimensional setting.
\end{abstract}

\maketitle

\section{Introduction}\label{introduction}

In this paper we consider the incompressible Navier-Stokes equations in $\Rd$, $d\geq4$. In the standard manner, we normalize to unit viscosity and project the system onto the space of divergence-free vector fields. This yields
\eq{\label{ns}
\dd_tu-\Delta u+\mathbb P\div u\otimes u=0
}
where $u:[0,T)\times\Rd\to\Rd$ is an unknown vector field and $\mathbb P=1-\Delta^{-1}\grad\div$ is the Leray projection. We will occasionally refer to the pressure $p=-\Delta^{-1}\div\div(u\otimes u)$ which is well-defined in $L_t^\infty L_x^\frac d2$ assuming that $u\in L_t^\infty L_x^d$. Note that $u\in L_t^\infty L_x^d$ is a natural setting to study regularity and blowup of \eqref{ns} because the associated norm is critical, i.e., $\|u\|_{L_t^\infty L_x^d}$ and \eqref{ns} are both invariant under the scaling $u(t,x)\mapsto\lambda u(\lambda^2t,\lambda x)$.

Let us consider, say, a finite-energy solution $u$ of \eqref{ns}  arising from divergence-free Schwartz initial data. If $u$ is smooth on $[0,T_*)$, we say that $u$ blows up at time $T_*$ if there exists no smooth extension to $[0,T_*]$. It is natural to ask what necessary conditions exist for $u$ to blow up at $T_*$, particularly those in  the form of a norm of $u$ becoming unbounded or diverging to $+\infty$. An early result of this type is due to Leray \cite{leray} who proved (originally for only $d=3$) that if $T_*$ is the blowup time, then
\eqn{
\liminf_{t\,\uparrow\, T_*}(T_*-t)^{\frac12-\frac d{2p}}\|u(t)\|_{L_x^p(\Rd)}\geq c
}
where $c=c(d,p)>0$ is an absolute constant and $d<p\leq\infty$. In a similar spirit, the classical work of Prodi, Serrin, and Ladyzhenskaya \cite{prodi,serrin,lady} implies a blowup criterion in which the pointwise lower bound is replaced by an average in time, in such a way that the spacetime norm becomes critical. Specifically, if $T_*$ is the blowup time, then
\eqn{
\int_0^{T_*}\|u(t)\|_{L_x^p(\Rd)}^{2/(1-\frac dp)}dt=\infty.
}
Both of these results degenerate at the $p=d$ endpoint which raises the natural question of whether the critical norm $\|u\|_{L_x^d(\Rd)}$ must become unbounded or diverge as $t\,\uparrow\,T_*$. This is a more difficult question in view of the fact that the spatial norms $\|u\|_{L_x^p}$, $p>d$ are subcritical; thus if $u$ concentrates at increasingly fine scales, one should expect these norms to become large. On the other hand, for the critical norm to grow, it is more likely due to many small concentrations of the solution scattered throughout space, not just the blowup profile itself. For this reason, one needs a genuinely different strategy to address the case $p=d$.

The breakthrough on this problem came in the paper of Escauriaza, Seregin, and \v{S}ver\'ak \cite{ess} who proved the three-dimensional result
\eqn{
\limsup_{t\,\uparrow\,T_*}\|u(t)\|_{L_x^3(\mathbb R^3)}=\infty
}
which was later improved to a pointwise limit as $t\uparrow T_*$ by Seregin \cite{seregin} (which, notably, is still open in the case $d\geq4$). The most important new tool was the backward uniqueness for the heat equation proved in \cite{bu} which is applied to a weak limit of solutions which ``zoom in'' on a potential blowup point. In order to apply backward uniqueness, one needs
pointwise bounds on $u$ and $\grad u$ in a large region away from the blowup point which the authors obtained using a Caffarelli-Kohn-Nirenberg-type $\epsilon$-regularity lemma. This appears to be the main reason that \cite{ess} is limited to $d=3$.

In the general case $d\geq3$, Dong and Du \cite{dongdu} attempt to avoid this issue by defining a local scale-invariant quantity whose smallness, combined with $\|u\|_{L_t^\infty L_x^d}<\infty$, implies local regularity of the solution.\footnote{\label{ddfootnote}Unfortunately there is a gap in the proof of this result in \cite{dongdu} (Lemma 3.2), as pointed out later by Dong and Wang \cite{dw}. Let us emphasize that the present work uses different techniques and is not impacted by this error.} Tools inspired by theirs will prove useful in our quantitative setting as well; see Section \ref{smalltolarge} for details.

In contrast to the classical blowup criteria above, the results using the compactness approach such as \cite{ess,dongdu,seregin,phuc,gkp,albritton}, etc.\ do not imply any quantitative information on the blowup rate of any critical norm. Recently, Tao \cite{tao} gave what appears to be the first quantitative rate for a critical quantity, namely that
\eq{\label{3d}
\limsup_{t\,\uparrow\,T_*}\frac{\|u(t)\|_{L_x^3(\mathbb R^3)}}{(\log\log\log\frac 1{T_*-t})^{c}}=\infty
}
for $d=3$, where $c>0$ is an absolute constant. Instead of proceeding by contradiction as in \cite{ess} and similar works, the idea is to assume the solution is concentrating near a blowup and show directly---using, for example, Carleman estimates instead of their qualitative derivatives like backward uniqueness---that there must be mild $L_x^3$ concentrations at an increasing number of spatial scales. Subsequent authors have used similar schemes to extend \eqref{3d}; for instance Barker and Prange quantify the improvement in \cite{seregin} and prove a theorem for type-1 blowups \cite{barkerprange}, while the author  shows that the denominator in Theorem \ref{blowupthm} can be replaced by $(\log\log\frac1{T_*-t})^c$ in the case that $u$ is axisymmetric \cite{p}.

In this work we are able to prove a quantitative blowup rate analogous to \eqref{3d} for $d\geq4$, answering a question of Tao, see Remark 1.6 in \cite{tao}. As in \cite{tao}, we assume for convenience that $u$ is a classical solution, meaning it is smooth with derivatives in $L_t^\infty L_x^2([0,T]\times\Rd)$. Since our results depend quantitatively on only $\|u\|_{L_t^\infty L_x^d}$, they can in principle be extended, for instance to the Leray-Hopf class as in \cite{ess}.

\begin{thm}\label{blowupthm}
Suppose $u$ is a classical solution of \eqref{ns} that blows up at $t=T_*$ and $d\geq4$. Then
\eqn{
\limsup_{t\,\uparrow\,T_*}\frac{\|u(t)\|_{L_x^d(\mathbb R^d)}}{(\log\log\log\log\frac 1{T_*-t})^{c}}=\infty
}
for a constant $c=c(d)>0$ depending only on the dimension.
\end{thm}

This is a straightforward consequence of our other main theorem which asserts that a solution satisfying the critical bound
\eq{\label{critical}
\|u\|_{L_t^\infty L_x^d([0,T]\times\Rd)}\leq A
}
is regular; in particular we can quantify its subcritical norms in terms of $A$. Let us take $A$ to be at least 2.

\begin{thm}\label{regularitythm}
If $u$ is a classical solution of \eqref{ns} on $[0,T]$ satisfying \eqref{critical} with $d\geq4$, then
\eqn{
\|\grad^ju(t)\|_{L_x^\infty(\Rd)}&\leq \exp\exp\exp\exp(A^{C})t^{-\frac{1+j}2}
}
for $t\in(0,T]$, where $C=C(j,d)$ depends only on $j\geq0$ and the dimension.
\end{thm}

\begin{remark}\label{remark}
Using ideas from \cite{p}, particularly Proposition 8, it is possible to improve the bounds in Theorems \ref{blowupthm} and \ref{regularitythm} if some mild symmetry assumptions are made on $u$. For example, suppose $u$ is axisymmetric\footnote{By this we mean the following: when regarded in the coordinate system which consists of polar coordinates $(r,\theta)$ in the $x_1,x_2$-plane and Cartesian coordinates in the rest, we have $u(x)=R_\theta(u(R_{-\theta}x))$ where $R_\theta$ denotes counterclockwise rotation by $\theta$ in the $x_1,x_2$-plane.} about the $x_3,x_4,\ldots,x_d$-plane. When $d=4$, one $\log$ and one $\exp$ can be removed from Theorems \ref{blowupthm} and \ref{regularitythm} respectively. When $d\geq5$, we may remove two $\log$s and two $\exp$s. In the latter case, in the proof of Proposition \ref{prop1}, we find the desired concentration at length scale $\ell=A^{-O(1)}$ using the slightly improved energy bound \eqref{maximalregularity}, while when $d=4$, we resort to pigeonholing the energy over $A^{O(1)}$-many length scales which yields an $\ell$ as small as $\exp(-A^{O(1)})$. An argument similar to Proposition 8 in \cite{p} allows one to avoid losing additional exponentials when locating annuli of regularity as in Proposition \ref{annuli}.
\end{remark}

Let us summarize why the approach in \cite{tao} breaks down in greater than three dimensions. The first set of difficulties arises when one would use the ``bounded total speed'' property, i.e., control on $\|u\|_{L_t^1L_x^\infty}$, see Proposition 3.1(ii) in \cite{tao}. One expects (for example, based on the heuristics following Proposition 9.1 in \cite{othertao}) that this property fails when $d\geq4$. In other words, one cannot expect any kind of ``speed limit'' for elements convected by $u$. Instead, we derive a procedure to propagate concentrations of the velocity and pressure from fine to coarse scales, encapsulated in Proposition \ref{nested}, which is a quantitative version of Lemma 3.2 in \cite{dongdu}. From this we can extract several important results including an $\epsilon$-regularity criterion (Proposition \ref{regularity}) and the backward-propagation lemma (Proposition \ref{bp}).

The second and more significant challenge in high dimensions is due to the lack of quantitative epochs of regularity as in Proposition 3.1(iii) in \cite{tao}. In the qualitative analysis, it suffices to use epochs of regularity for which one has absolutely no lower bound on the length, nor any explicit upper bound on $|u|$, $|\grad u|$, etc. (For example, see the use of Proposition 2.4 in the proof of Proposition 5.3 in \cite{dongdu}.) This becomes a problem when one needs to propagate concentrations of vorticity through space and into a distant annulus of regularity, as the width of the time interval on which one has regularity determines the lower bound one can extract from unique continuation for the heat equation. We will remedy this by substituting spacetime partial regularity in place of epochs of regularity. This creates some new difficulties; first that when one propagates a high frequency concentration of the solution backward in time, a priori there is no guarantee that the resulting concentration has any of its $L_{t,x}^2$ mass inside the regular region. There is a particular fractal arrangement of concentrations in spacetime which is consistent with this obstruction; indeed the objective of Proposition \ref{prop1} is to locate a scale where we may rule it out.

The second difficulty faced when propagating the vorticity using only partial regularity is the following: the usual Carleman inequality for unique continuation has as its domain a large ball in space (compared to the length of the time interval); however we wish to propagate the vorticity for a great distance through a thin spacetime slice. We are able to accomplish this without the bounds suffering too badly (losing only one additional exponential compared to the $d=3$ case) by repeatedly applying the Carleman inequality in a series of moving and expanding balls lying in an expanding slice of spacetime. We show that the iteration of unique continuation accelerates exponentially away from the initial vorticity concentration. The positive feedback loop this creates is essential for arriving at the claimed bounds, as unique continuation through a uniformly thin slice would lead to an unbounded number of logarithms and exponentials in Theorems \ref{blowupthm} and \ref{regularitythm}.

The plan of the paper is as follows: in Section \ref{preliminaries}, we give some preliminaries including a useful decomposition for estimating solutions in high-integrability spaces. In Section \ref{tools}, we prove some straightforward energy estimates, then introduce our quantitative analogue of the Dong-Du lemma for propagating concentrations of the solution from fine to coarse scales. Then we apply it to partial regularity and backward propagation of high frequency concentrations. In Section \ref{carleman} we quote the high-dimensional version of the Carleman inequality from \cite{tao} for quantitative unique continuation, then show the iteration through a slice expanding in space. In Section \ref{mainprop}, we prove the main propositions: first the back propagation into a regular cylinder, then the successive use of Carleman inequalities to propagate it forward to the final time. Finally, in Section \ref{proofoftheorems}, we use Propositions \ref{prop1} and \ref{prop2} to prove Theorems \ref{blowupthm} and \ref{regularitythm}.

\section{Preliminaries}\label{preliminaries}

\subsection{Notation}\label{notation}

We use asymptotic notation $X\lesssim Y$ or $X=O(Y)$ to mean that there is a constant $C(d)$ depending only on the spatial dimension such that $|X|\leq C(d)Y$. Moreover $X\sim Y$ is an abbreviation for $X\lesssim Y\lesssim X$. A subscript on $\lesssim$ or $O$ indicates that the constant may depend on additional parameters. As in \cite{tao}, we define the hierarchy of powers $A_j=A^{C_0^j}$ where $A$ is as in \eqref{critical} and $C_0$ is a large constant which depends only on the dimension $d$. Throughout the arguments we will freely enlarge $A$ so that $A\geq C_0$ and $C_0$ so that it defeats any constants in the asymptotic notation.

If $I\subset\mathbb R$ is a time interval, we use $|I|$ to denote its length. If $\Omega\subset\Rd$, $|\Omega|$ will denote its $d$-dimensional Lebesgue measure. For a finite set $E$, we denote its cardinality by $\#(E)$. If $x_0\in\Rd$ and $R>0$, we will write $B(x_0,R)$ to denote the closed ball $\{x\in\Rd:|x-x_0|\leq R\}$. If $z_0=(t_0,x_0)\in\mathbb R\times\Rd$ is a spacetime point, define the parabolic cylinder $Q(z_0,R):=[t_0-R^2,t_0]\times B(x_0,R)$ and $Q(R):=Q(0,R)$. If $B=B(x_0,R)$ and $Q=Q(z_0,R)$ where $z_0=(t_0,x_0)$, we define the dilations $\lambda B:=B(x_0,\lambda R)$ and $\lambda Q:=Q(z_0,\lambda R)$.

For vectors $u,v\in\Rd$, define the tensor products
\eqn{
(u\otimes v)_{ij}:=u_iv_j,\quad u\odot v:=\frac12(u\otimes v+v\otimes u),\quad u^{\otimes2}:=u\otimes u.
}
We also make use of the Frobenius inner product $A:B:=\tr(A^tB)$.

For $\Omega\subset\mathbb R^n$ and $I\subset\mathbb R$, we will use the Lebesgue norms
\eqn{
\|f\|_{L_x^q(\Omega)}:=\left(\int_\Omega|f(x)|^qdx\right)^{1/q}
}
and
\eqn{
\|f\|_{L_t^pL_x^q(I\times\Omega)}:=\left(\int_I\|f(t,\cdot)\|_{L_x^q(\Omega)}^pdt\right)^{1/p}
}
with the usual modifications if $p=\infty$ or $q=\infty$. When $p=q$ we use the abbreviation $L_{t,x}^p:=L_t^pL_x^p$.

For a Schwartz function $f:\Rd\to\mathbb R^n$, we define the Fourier transform
\eqn{
\hat f(\xi)=\int_{\Rd}e^{-ix\cdot\xi}f(x)dx
}
and the Littlewood-Paley projection by the formula
\eqn{
\widehat{P_{\leq N}f}(\xi):=\varphi(\xi/N)\hat f(\xi)
}
where $\varphi:\mathbb R^3\to\mathbb R$ is a radial bump function supported in $B(0,1)$ such that $\varphi\equiv1$ in $B(0,1/2)$. Then let
\eqn{
P_N&:=P_{\leq N}-P_{\leq N/2},\quad P_{>N}:=1-P_{\leq N},\quad P_{N\leq\,\cdot\,\leq M}:=P_{\leq M}-P_{\leq N/2},
}
etc. Any sums indexed by capital letters such as $\sum_N$ or $\sum_{N>A}$ should be taken to have indicies ranging over the dyadic integers $2^\mathbb Z$.

\subsection{Bernstein-type inequalities}

The following bounds on frequency-localized Fourier multipliers will prove useful, see Lemma 2.1 in \cite{tao}.

\begin{lem}\label{bernstein}
Let $m\in C^\infty(\Rd\to\mathbb C)$ be a multiplier supported in $B(N)$ for some frequency $N>0$, obeying
\eqn{
|\grad^jm(\xi)|\leq MN^{-j}
}
for all $0\leq j\leq 100d$ for some $M>0$. Then the Fourier multiplier $\widehat{T_mf}(\xi):=m(\xi)f(\xi)$ satisfies the bound
\eqn{
\|T_mf\|_{L^q(\Rd)}\lesssim_{p,q} MN^{\frac dp-\frac dq}\|f\|_{L^p(\Rd)}
}
assuming $1\leq p\leq q\leq\infty$ and $f\in L^p(\Rd)$ is smooth.
\end{lem}

Let us record a useful application of Lemma \ref{bernstein}. Taking $P_Ne^{t\Delta}$ and summing, we obtain
\eq{\label{heat}
\|\grad^je^{t\Delta}f\|_{L_x^q}&\lesssim_{p,q,j} t^{-\frac12(\frac dp-\frac dq+j)}\|f\|_{L_x^p}
}
for $t>0$, $j\geq0$, and $1\leq p\leq q\leq\infty$. Moreover, combining Lemma \ref{bernstein} with \eqref{ns} and \eqref{critical}, the following bounds on the frequency-localized vector fields are immediate.

\begin{lem}\label{frequencylocalized}
If $u$ solves \eqref{ns} on $[-T,0]$ and admits the bound \eqref{critical}, then we have
\eqn{
\|\grad^jP_Nu\|_{L_{t,x}^\infty([-T,0]\times\Rd)}\lesssim_j AN^{1+j},\quad\|\dd_tP_Nu\|_{L_{t,x}^\infty([-T,0]\times\Rd)}\lesssim A^2N^3
}
for all $j\geq0$, $N>0$.
\end{lem}

As in \cite{tao}, Lemma \ref{bernstein} has a spatially localized version.

\begin{lem}\label{localbernstein}
Let $m$, $N$, and $M$ be as in Lemma \ref{bernstein}, $\Omega\subset\Rd$ open, $B\geq1$, $K\geq10$, and $\Omega_{B/N}:=\{x\in\Rd:\dist(x,\Omega)<B/N\}$. Then we have
\eqn{
\|T_mf\|_{L^{q_1}(\Omega)}&\lesssim_{p_1,p_2,q_1,q_2,K} MN^{\frac d{p_1}-\frac d{q_1}}\|f\|_{L^{p_1}(\Omega_{B/N})}+B^{-K}M|\Omega|^{\frac1{q_1}-\frac1{q_2}}N^{\frac d{p_2}-\frac d{q_2}}\|f\|_{L^{p_2}(\Rd)}
}
assuming $1\leq p_1\leq q_1\leq\infty$, $1\leq p_2\leq q_2\leq\infty$, and $q_2\geq q_1$.
\end{lem}
If $f$ admits a decomposition $\sum f_i$, the same proof found in \cite{tao} allows the second term on the right-hand side in the Bernstein inequality to be estimated separately for each $f_i$, each with its own choice of $p_2$ and $q_2$. (This will be useful when paired with Proposition \ref{decomposition}.)

In a similar spirit, the following simple result will be useful.
\begin{lem}\label{localize}
If $N,K>0$, $j\geq0$, $p\leq q$, $0<r_1<r_2$, $f\in C^\infty(\Rd)$, and $\phi\in C_c^\infty(\Rd)$ with $\phi\equiv1$ in $B(r_2)$, then
\eqn{
\|P_N\grad^jf\|_{L^p(B(r_1))}&\lesssim_{r_1,r_2,p,q,j,K,\phi}\|P_N(\phi \grad^jf)\|_{L^p(B(r_1))}+N^{-K}\|f\|_{L^q(\Rd)}.
}
\end{lem}

\begin{proof}
With $\psi(\xi)$ the Fourier multiplier for $P_1$, we have
\eqn{
P_N\grad^jf(x)&=\int_{\Rd}\check\psi(y)(\phi \grad^jf)(x-y/N)dy+\int_{\Rd\setminus B(cN)}\check\psi(y)((1-\phi)\grad^jf)(x-y/N)dy
}
as long as $x$ is restricted
to $B(r_1)$ and $c$ is chosen sufficiently small compared to $r_2-r_1$. The first term is exactly $P_N(\phi\grad^jf)(x)$ and the second term is straightforward to estimate using integration by parts, polynomial decay of $\check\psi$ and its derivatives, and H\"older's inequality.
\end{proof}

\subsection{Sharp-flat decomposition of the solution}

A difficulty of working in $L_x^d$ is that while one would wish to make use of energy methods, the solution does not have enough decay to be in any $L_x^2$-based spaces. In the cases $d=3,4$ one can avoid this problem by some manner of splitting $u$ into one flow solving a linear equation and another that solves a complementary nonlinear equation, see \cite{calderon,tao}. For example, the method in \cite{tao} of considering $u(t)-e^{(t-t_0)\Delta}u(t_0)$, i.e., removing the heat flow part of the evolution, leaves the remaining nonlinear flow in $L_t^\infty L_x^p$ for $p\in[\frac d2,d]$. Unfortunately when $d\geq5$, this range excludes the important energy space $L_t^\infty L_x^2$.

In the general case $d\geq3$ we address this difficulty using the following decomposition of $u$ very similar to the one in \cite{p}. We remark that decompositions based on a Picard-type iteration in the same spirit have also appeared in \cite{cp,gkp,ab}. The idea is essentially to subtract off a Picard iterate starting from an initial condition $u(t_0)$. The critical bound \eqref{critical} implies good subcritical estimates on the iterate thanks to smoothing from the heat propagator, and one can show inductively using Duhamel's formula that the difference lies in lower integrability spaces including $L_t^\infty L_x^2$. Moreover, the difference satisfies a Navier-Stokes-type equation which leads to estimates that will be useful later.

\begin{prop}\label{decomposition}
Suppose $u$ is a classical solution of \eqref{ns} on $[-T,0]$ with the bound \eqref{critical}. Then for every $T_1\in[0,T/2]$, there exist $\ulin$ and $\unlin$ such that the following hold:
\begin{itemize}
    \item We have the decomposition
    \eqn{
    u=\ulin+\unlin\text{ on }[-T_1,0].
    }
    \item If $d\leq p\leq\infty$ and $j\geq0$, then
    \eq{
    \|\grad^j\ulin\|_{L_t^\infty L_x^p([-T_1,0]\times\Rd)}&\leq A^{O_j(1)}T_1^{-\frac12(1+j-\frac dp)},\label{ulin}\\
    \|P_N\ulin\|_{L_{t,x}^\infty([-T_1,0]\times\Rd)}&\leq A^{O(1)}e^{-T_1N^2/O(1)}T_1^{-\frac12}.\label{PNuflat}
    }
    \item If $1\leq p\leq d$ and $1<q<\infty$, then\footnote{We thank the referee for bringing \eqref{maximalregularity} to our attention which allows a simplification to the argument.}
    \eq{
    \|\unlin\|_{L_t^\infty L_x^p([-T_1,0]\times\Rd)}&\leq A^{O(1)}T_1^{\frac12(\frac dp-1)},\label{unlin}\\
    \|\grad\unlin\|_{L_{t,x}^2([-T_1,0]\times\Rd)}&\leq A^{O(1)}T_1^{\frac d4-\frac 12},\label{sharpenergy}\\
    \|\grad\unlin\|_{L_t^qL_x^\frac d2([-T_1,0]\times\Rd)}&\lesssim_q A^{O(1)}T_1^{\frac1q}.\label{maximalregularity}
    }
    \item $\unlin$ solves
    \eq{\label{sharpequation}
    \dd_t\unlin+\mathbb P\div(\unlin\otimes\unlin+2\ulin\odot\unlin)-\Delta\unlin=f
    }
    where $f$ obeys estimates
    \eq{\label{force}
    \|\grad^jf\|_{L_t^\infty L_x^p([-T_1,0]\times\Rd)}&\leq A^{O_j(1)}T_1^{-\frac12(3+j-\frac dp)}
    }
    for $\frac d2\leq p\leq\infty$ and $j\geq0$.
\end{itemize}
\end{prop}

\begin{proof}

Starting with
\eqn{
\ulin_0:=0,\quad\unlin_0:=u,
}
we inductively define for $n\geq1$
\eqn{
\ulin_n(t)&:=e^{(t-\tau_{n-1})\Delta}u_{n-1}(\tau_{n-1})-\int_{\tau_{n-1}}^te^{(t-t')\Delta}\mathbb P\div\ulin_{n-1}\otimes\ulin_{n-1}(t')dt',\\
\unlin_n(t)&:=-\int_{\tau_{n-1}}^te^{(t-t')\Delta}\mathbb P\div(u\otimes u-\ulin_{n-1}\otimes\ulin_{n-1})(t')dt'
}
where we have chosen a sequence of $O(1)$-many times $-2T_1<\tau_1<\tau_2<\cdots<-T_1$ such that $\tau_i-\tau_{i+1}=T_1/O(1)$. We prove \eqref{ulin} on the shrinking time intervals $[\tau_n,0]$ with $\ulin$ replaced by $\ulin_n$ by induction on $n$. For $n=0$ it is trivial. Suppose the claim for some $n-1\geq0$. Then, for $t\in[\tau_n,0]$,
\eqn{
\|\grad^j\ulin_n(t)\|_{L_x^p(\Rd)}&\lesssim(t-\tau_{n-1})^{-\frac12(1+j-\frac dp)}A\\
&\quad+\int_{\tau_{n-1}}^t(t-t')^{-\frac12}\|\ulin_{n-1}(t')\|_{L_{x}^p(\Rd)}\|\grad^j\ulin_{n-1}(t')\|_{L_{x}^\infty(\Rd)}dt'
}
which gives the desired bound. Then \eqref{PNuflat} follows similarly by induction using Duhamel's principle and a paraproduct decomposition.

Next, it is convenient to decompose $\unlin_n=\unlinI_n+\unlinII_n$ where
\eqn{
\unlinI(t)&:=-2\int_{\tau_{n-1}}^te^{(t-t')\Delta}\mathbb P\div\ulin_{n-1}\odot\unlin_{n-1}(t')dt',\\
\unlinII(t)&:=-\int_{\tau_{n-1}}^te^{(t-t')\Delta}\mathbb P\div\unlin_{n-1}\otimes\unlin_{n-1}(t')dt'.
}
We claim the bound in \eqref{unlin} for $\unlin_n$, specifically in the range $\max(\frac d{n+1},1)\leq p<d$ and on the time interval $[\tau_n,0]$. Thus we will obtain the desired result by taking $n$ large depending on $d$. Note that the $p=d$ case is immediate from \eqref{critical} and \eqref{ulin}. As a base case, we consider $n=1$ for which $\unlinI_1=0$. For $\unlinII_1$,
\eqn{
\|\unlinII_1(t)\|_{L_x^p(\Rd)}&\lesssim\int_{\tau_{n-1}}^t(t-t')^{-\frac12(3-\frac dp)}\|u(t')\|_{L_x^d(\Rd)}^2dt'
}
which yields the desired result using \eqref{critical} assuming $\frac d2\leq p<d$. Now assume the desired inequality for some $n-1\geq1$. Then
\eqn{
\|\unlinI(t)\|_{L_x^p(\Rd)}&\lesssim\int_{\tau_{n
-1}}^t(t-t')^{-\frac12(1+\frac ds+\frac dr-\frac dp)}\|\ulin_{n-1}(t')\|_{L_x^s(\Rd)}\|\unlin_{n-1}(t')\|_{L_x^r(\Rd)}dt',
}
assuming $\frac1p\leq\frac1s+\frac1r$. This is integrable in time, and furthermore we can apply \eqref{ulin} and \eqref{unlin}, by taking $r=\frac{dp}{d-p}$ and $s=d$, and assuming additionally that $\max(\frac d{n+1},1)\leq p<\frac d2$. If instead we take $r=\frac d2$ and $\frac1s=\max(\frac1p-\frac2d,0)$, we obtain the same result but instead for $\frac d3\leq p<d$. Combining these, we have the full range of $p$. Next we consider $\unlinII$. With $\frac1r=\frac12(\frac1d+\frac1p)-\epsilon$,
\eqn{
\|\unlinII_n(t)\|_{L_x^p(\Rd)}&\lesssim\int_{\tau_{n-1}}^t(t-t')^{-1+\epsilon d}\|\unlin_{n-1}(t')\|_{L_x^r(\Rd)}^2dt'
}
implies the desired bound upon taking $\epsilon$ sufficiently small depending on $p$ and $d$. \eqref{ulin}-\eqref{sharpenergy} therefore hold upon setting $\ulin,\unlin:=\ulin_d,\unlin_d$.

One readily computes \eqref{sharpequation} with $f=\mathbb P\div(\ulin_{d-1}\otimes\ulin_{d-1}-\ulin_d\otimes\ulin_d)$. Then \eqref{force} follows by H\"older's inequality and \eqref{ulin}. Multiplying \eqref{sharpequation} by $\unlin$ and integrating over $\Rd$, we have
\eqn{
\frac d{dt}\int_{\Rd}\frac{|\unlin|^2}{2}dx&=-\int_{\Rd}|\grad\unlin|^2-\int_{\Rd}\unlin\cdot(\unlin\cdot\grad\ulin)+\unlin\cdot f
}
and therefore we can apply \eqref{unlin}, \eqref{ulin}, and \eqref{force} to find
\eqn{
\|\grad\unlin\|_{L_{t,x}^2([-T_1,0]\times\Rd)}&\lesssim\|\unlin\|_{L_t^\infty L_x^2([-T_1,0]\times\Rd)}+\|\unlin\|_{L_t^\infty L_x^2([-T_1,0]\times\Rd)}\|\grad\ulin\|_{L_{t,x}^\infty([-T_1,0]\times\Rd)}\\
&\quad+\|\unlin\|_{L_t^\infty L_x^1([-T_1,0]\times\Rd)}\|f\|_{L_{t,x}^\infty([-T_1,0]\times\Rd)}
}
which proves \eqref{sharpenergy}.

Finally, we note that
\eqn{
(\dd_t-\Delta)\grad\unlin=-\grad\mathbb P\div(u\otimes u-\ulin_{d-1}\otimes\ulin_{d-1}).
}
Thus by \eqref{heat} and maximal regularity for the heat equation,
\eqn{
\|\grad\unlin\|_{L_t^qL_x^\frac d2([-T_1,0]\times\mathbb R^d)}\lesssim_q T_1^{-\frac12+\frac1q}\|\unlin(-T_1)\|_{L_x^\frac d2(\Rd)}+\|u\otimes u-\ulin_{d-1}\otimes\ulin_{d-1}\|_{L_t^qL_x^\frac d2([-T_1,0]\times\Rd)}.
}
We conclude \eqref{maximalregularity} by \eqref{ulin}, \eqref{critical}, and H\"older's inequality.

\end{proof}

\section{Tools for controlling spacetime concentrations}\label{tools}

\subsection{Local energy estimates}

We will make use of three slightly different consequences of the local energy equality for \eqref{ns}. The second is an extension of Lemma 2.2 in \cite{dongdu}, now with the dependence on $A$ made explicit. (See below for the definitions of $C$ and $D$.)

\begin{lem}
Let $u$ be a smooth solution of \eqref{ns} satisfying \eqref{critical} on $[-T,0]$, $r>0$, and $I\subset[-T,0]$. Then we have
\eq{\label{firstlocalenergy}
\sup_I\int_{B(r)}|u|^2dx+\int_I\int_{B(r)}|\grad u|^2dxdt\leq A^{O(1)}r^{d-4}|I|,
}
\eq{\label{localCD}
\|u\|_{L_t^\infty L_x^2(Q(z_0,r/2))}+\|\grad u\|_{L_{t,x}^2(Q(z_0,r/2))}\lesssim r^{\frac d2-1}(C(r,z_0)+D(r,z_0))A^{1/2},
}
and, for $t_0\leq t\leq t_0+10r^2$,
\begin{equation}\begin{aligned}\label{energychange}
&\int_{Q(z_0,r/2)}\frac{|u(t)|^2}2dx-\int_{Q(z_0,r/2)}\frac{|u(t_0)|^2}2dx\\
&\quad\quad\lesssim\|\grad u\|_{L_{t,x}^2(Q(z_0,r))}A^2r^{\frac d2-1}+D(Q(r,z_0))Ar^{d-2}.
\end{aligned}\end{equation}
\end{lem}

\begin{proof}
All three estimates are elementary applications of the local energy equality
\eqn{
\frac d{dt}\int_\Rd\frac{|u|^2}2\psi dx+\int_\Rd|\grad u|^2\psi dx=\int_\Rd\frac{|u|^2}2(\dd_t\psi+\Delta\psi+u\cdot\grad\psi)+pu\cdot\grad\psi dx,
}
along with H\"older's inequality, \eqref{critical}, integration by parts, and the Calder\'on-Zygmund estimate for the pressure.
\end{proof}

\subsection{Propagation from small to large scales}\label{smalltolarge}

Define the local scale-invariant quantities
\eqn{
C(R,z_0):=R^{-\frac d2-\frac1{d+3}+1}\|u\|_{L_{t,x}^{2(d+3)/(d+1)}(Q(z_0,R))}
}
and
\eqn{
D(R,z_0):=R^{-\frac d2-\frac1{d+3}+1}\|p\|_{L_{t,x}^{(d+3)/(d+1)}(Q(z_0,R))}^{1/2}.
}
For brevity, if $Q=Q(z_0,R)$, we write $C(Q)$ in place of $C(R,z_0)$. These quantities appear in \cite{dongdu}, although here we have defined them slightly differently so they are proportional to the norm of $u$. The following proposition is closely related to Proposition 3.1 in \cite{dongdu}. Their method of proof is by contradiction and uses a compactness argument to find suitable values $\eta$ and $\epsilon$. Thus such an approach does not give any information on how they depend on $A$; see also the comments in footnote \ref{ddfootnote}.

\begin{prop}\label{nested}
Let $u$ be a smooth solution of \eqref{ns} satisfying \eqref{critical}. Then for any $\epsilon\leq A^{-d^3}$, if $z_0\in[-T/2,0]\times\Rd$, $\rho\leq T/4$, and
\eqn{
C(\rho,z_0)+D(\rho,z_0)\leq\epsilon,
}
then
\eqn{
C(r,z_1)+D(r,z_1)\leq \epsilon A^{O(1)}
}
for any $z_1\in Q(z_0,\rho/2)$ and $r\in(0,\rho/2)$.
\end{prop}

As in \cite{dongdu}, Proposition \ref{nested} is obtained by iteratively applying Lemma \ref{step} below. The point is that given a lower bound $C(r,z_1)+D(r,z_1)>\epsilon$ in a small cylinder, the lemma implies the same lower bound in a cylinder dilated by a factor of $A$. This step can be iterated until it yields a cylinder $Q'$ that is comparable in length to $Q(z_0,\rho)$, the ratio depending on $A$. Since $Q'$ can be smaller than $Q(z_0,\rho)$, the scaling factors in the definition of $C$ and $D$ lead to the loss of $A^{O(1)}$.

\begin{lem}\label{step}
Let $\epsilon$, $u$, $\rho$, and $z_0$ be as in Proposition \ref{nested}. Then, with $\eta=A^{-1}$,
\eq{\label{start}
C(\rho,z_0)+D(\rho,z_0)\leq\epsilon,
}
implies
\eq{\label{finish}
C(\eta\rho,z_0)+D(\eta\rho,z_0)\leq\epsilon.
}
\end{lem}

\begin{proof}
We translate and rescale so that $z_0=0$ and $\rho=1$. By \eqref{localCD},
\eq{\label{epsilonenergy}
\|u\|_{L_t^\infty L_x^2(Q(\frac12))}+\|\grad u\|_{L_{t,x}^2(Q(\frac12))}\lesssim\epsilon A^\frac12. 
}
Fix the large frequency scale $N=\epsilon^{-\frac1d}$. By interpolation and Lemma \ref{localbernstein},
\eqn{
\|P_{>N}u\|_{L_{t,x}^{2+\frac4{d+1}}(Q(\eta))}&\leq\sum_{M>N}\|P_Mu\|_{L_t^2L_x^{2\frac{d+3}{d+1}}(Q(\eta))}^{\frac{d+1}{d+3}}\|P_Mu\|_{L_t^\infty L_x^{2\frac{d+3}{d+1}}(Q(\eta))}^{\frac 2{d+3}}\\
&\lesssim\sum_{M>N}(M^{-\frac 3{d+3}}\|P_M\grad u\|_{L_{t,x}^2(Q(\frac14))}+M^{-100d^3})^{\frac{d+1}{d+3}}\\
&\quad\quad\times(M^{\frac d{d+3}}\|P_Mu\|_{L_t^\infty L_x^2(Q(\frac14))}+M^{-100d^3})^{\frac2{d+3}}.
}
Fix a spatial cutoff $\varphi\in C_c^\infty(B(1/2))$ with $\varphi\equiv1$ in $B(1/3)$. By Lemma \ref{localize}, Plancherel, and \eqref{epsilonenergy},
\eqn{
\sum_{M>N}\|P_Mu\|_{L_t^\infty L_x^2(Q(\frac14))}^2&\lesssim\sum_{M>N}\left(\|P_M(\varphi u)\|_{L_t^\infty L_x^2(Q(\frac14))}^2+M^{-100d^3}A^2\right)\\
&\lesssim(\epsilon^2+N^{-100d^3})A^2
}
and by the same reasoning
\eqn{
\sum_{M>N}\|P_M\grad u\|_{L_{t,x}^2(Q(\frac14))}^2&\lesssim(\epsilon^2+N^{-100d^3})A^2.
}
Therefore, by H\"older's inequality, the main term is
\eqn{
&\sum_{M>N}M^{-\frac1{d+3}}\|P_M\grad u\|_{L_{t,x}^2(Q(\frac14))}^{\frac{d+1}{d+3}}\|P_Mu\|_{L_t^\infty L_x^2(Q(\frac14))}^{\frac2{d+3}}\\
&\quad\lesssim N^{-\frac1{d+3}}\left(\sum_{M>N}\|P_M\grad u\|_{L_{t,x}(Q(\frac14))}^2\right)^{\frac12\frac{d+1}{d+3}}\left(\sum_{M>N}\|P_Mu\|_{L_t^\infty L_x^2(Q(\frac14))}^2\right)^{\frac1{d+3}}\\
&\quad\lesssim N^{-\frac1{d+3}}\epsilon+N^{-50d^2}A.
}
The remaining terms all involve the small global Bernstein error and can be estimated similarly to find
\eqn{
\|P_{>N}u\|_{L_{t,x}^{2+\frac4{d+1}}(Q(\eta))}&\lesssim N^{-\frac1{d+3}}\epsilon+N^{-50d^2}A\lesssim C_0^{-1}\eta^{\frac{2+d}{2+\frac4{d+1}}-1}\epsilon.
}

To study the low frequencies, let $\phi\in C_c^\infty(Q(1/2))$ be a spacetime cutoff function satisfying $\phi\equiv1$ in $Q(1/3)$. Using Duhamel's formula we decompose $u$ into local and global parts,
\newcommand\uloc{u^l}
\newcommand\uglob{u^g}
\eqn{
\uloc(t):=e^{(t+1)\Delta}((\phi u)(-1))-\int_{-1}^te^{(t-t')\Delta}\div(\phi( u\otimes u+p\Id))(t')dt',\quad\uglob:=u-\uloc.
}
By H\"older, Lemma \ref{bernstein}, \eqref{heat}, \eqref{epsilonenergy}, fractional integration, and \eqref{start},
\eqn{
\|P_{\leq N}\uloc\|_{L_{t,x}^{2+\frac4{d+1}}(Q(\eta))}
&\lesssim\eta^{\frac{d+2}{2+\frac4{d+1}}}\|e^{(t+1)\Delta}((\phi u)(-1))\|_{L_{t,x}^\infty([-\eta^2,0]\times\Rd)}\\
&\hspace{-.55in}+N^{\frac d2+\frac1{d+3}-2}\bigg\|\int_{-1}^te^{(t-t')\Delta}\div(\phi(u\otimes u+p\Id))(t')dt'\bigg\|_{L_t^{2+\frac4{d+1}}L_x^{\frac{d(d+3)}{d^2-5}}([-\eta^2,0]\times\Rd)}\\
&\hspace{-.65in}\lesssim\eta^{\frac{d+2}{2+\frac4{d+1}}}\epsilon A^\frac12+N^{\frac d2+\frac1{d+3}-1}\epsilon^2\lesssim C_0^{-1}\eta^{\frac{2+d}{2+\frac4{d+1}}-1}\epsilon.
}
Next observe that $P_{\leq N}\uglob$ solves the heat equation in $Q(1/3)$ so by H\"older's inequality and well-known parabolic theory,
\eqn{
\|P_{\leq N}\uglob\|_{L_{t,x}^{2+\frac4{d+1}}(Q(\eta))}&\lesssim\eta^{\frac{d+2}{2+\frac4{d+1}}}\|P_{\leq N}\uglob\|_{L_{t,x}^\infty(Q(\eta))}\lesssim\eta^{\frac{d+2}{2+\frac4{d+1}}}\|P_{\leq N}\uglob\|_{L_{t,x}^{2+\frac4{d+1}}(Q(\frac14))}.
}
Clearly $P_{\leq N}\uglob=u-P_{>N}u-P_{\leq N}\uloc$. The first piece can be estimated using \eqref{start}, while the other two we have already addressed. (Note that the estimates are unaffected by changing the domain to $Q(\frac14)$ except for the heat propagator part of $P_{\leq N}u^l$; however even the worse bound $\epsilon A^\frac12$ without the improvement from using H\"older on $Q(\eta)$ suffices.) In total,
\eqn{
C(\eta)&\lesssim C_0^{-1}\epsilon.
}
Next we consider the pressure. From the decomposition
\eqn{
P_{>N}p&=\Delta^{-1}\div\div P_{>N}\Big(2P_{\leq N/5}u\odot P_{>N/5}u+(P_{>N/5}u)^{\otimes2}\Big)=:\Pi_1+\Pi_2,
}
we have
\eqn{
\|\Pi_1\|_{L_{t,x}^{1+\frac2{d+1}}(Q(\eta))}&\lesssim\|P_{\leq N/5}u\|_{L_{t,x}^{2+\frac4{d+1}}(Q(2\eta))}\|P_{>N/5}u\|_{L_{t,x}^{2+\frac4{d+1}}(Q(2\eta))}+(\eta N)^{-50d^2}A^2
}
by Lemma \ref{localbernstein} and \eqref{critical}. By the same calculations by which we estimated $C$ above and the large choice of $N$, this implies
\eqn{
\|\Pi_1\|_{L_{t,x}^{1+\frac2{d+1}}(Q(\eta))}&\lesssim C_0^{-2}\eta^{\frac{d+2}{1+\frac2{d+1}}-2}\epsilon^2.
}
For the other term we have
\eqn{
\|\Pi_2\|_{L_{t,x}^{1+\frac2{d+1}}(Q(\eta))}&\lesssim\|P_{>N/5}u\|_{L_{t,x}^{2+\frac4{d+1}}(Q(\frac15))}^2+N^{-50d^2}A^2\lesssim C_0^{-2}\eta^{\frac{d+2}{1+\frac2{d+1}}-2}\epsilon^2
}
again by Lemma \ref{localbernstein} and the calculations above.

Next we turn to the low frequencies. With $\varphi\in C_c^\infty(Q(\frac15))$ a new spacetime cutoff satisfying $\varphi\equiv1$ in $Q(\frac16)$, define
\eqn{
p^l:=-\mathcal N\div\div (\varphi u\otimes u),\quad p^g:=p-p^l
}
where $\mathcal N$ is the Newton potential. To estimate the local contribution we employ the paraproduct decomposition
\eqn{
P_{\leq N}p^l&=-\mathcal N\div\div P_{\leq N}\bigg(\varphi(P_{\leq 5N}u)^{\otimes2}+\sum_{\substack{N'\sim N''\\\max(N',N'')>5N}}\varphi P_{N'}u\otimes P_{N''}u\bigg)=:\Pi_3+\Pi_4.
}
The calculations above imply that $P_{\leq N}u$ can be decomposed as $v+w$ where
\eqn{
\|v\|_{L_{t,x}^q(Q(\frac15))}\leq C_0^{-1}\eta^{-\frac9{10}}\epsilon,\quad\|w\|_{L_{t,x}^{2+\frac4{d+1}}(Q(\frac15))}\leq C_0^{-1}\eta^{\frac{d+2}{2+\frac4{d+1}}-1}\epsilon
}
for any $q\geq1$. (For example, let $w$ be the nonlinear part of $P_{\leq N}u^l$ and $v$ the rest.) Thus, using the Calder\'on-Zygmund estimate for $\mathcal N$,
\eqn{
\|\Pi_3\|_{L_{t,x}^{1+\frac2{d+1}}(Q(\eta))}&\lesssim\eta^{\frac{d+2}{1+\frac2{d+1}}-\frac1{10}}\|\varphi v\otimes v\|_{L_{t,x}^{2q}([-\eta^2,0]\times\Rd)}+\eta^{\frac{d+2}{2+\frac4{d+1}}}\|\varphi v\odot w\|_{L_{t,x}^{2+\frac4{d+1}}([-\eta^2,0]\times\Rd)}\\
&\quad+\|\varphi w\otimes w\|_{L_{t,x}^{1+\frac2{d+1}}([-\eta^2,0]\times\Rd)}\\
&\lesssim C_0^{-2}\eta^{\frac{d+2}{1+\frac2{d+1}}-2}\epsilon^2
}
where $q$ is taken large but finite to avoid the unboundedness of $\mathcal N\div\div$ at the endpoint. By the calculations for $P_{>N}u$,
\eqn{
\|\Pi_4\|_{L_{t,x}^{1+\frac2{d+1}}(Q(\eta))}&\lesssim\sum_{N'\gtrsim N}(N')^{-\frac2{d+3}}\epsilon^2+(N')^{-100d^2}A^2\leq C_0^{-2}\eta^{\frac{d+2}{1+\frac2{d+1}}-2}\epsilon^2.
}
Finally, observe that $P_{\leq N}p^g$ is harmonic in $Q(\frac16)$. Therefore
\eqn{
\|P_{\leq N}p^g\|_{L_{t,x}^{1+\frac2{d+1}}(Q(\eta))}&\lesssim\eta^{\frac{d}{1+\frac2{d+1}}}\|P_{\leq N}p^g\|_{L_t^{1+\frac2{d+1}}L_{x}^\infty(Q(\eta))}\lesssim\eta^{\frac{d}{1+\frac2{d+1}}}\|P_{\leq N}p^g\|_{L_{t,x}^{1+\frac2{d+1}}(Q(\frac16))}.
}
Then the decomposition $P_{\leq N}p^g=p-P_{>N}p-P_{\leq N}p^l$ along with the above estimates and \eqref{start} implies the desired bound. This completes the estimate of $D(\eta)$.
\end{proof}

\subsection{Annuli and slices of regularity}

The first application of Proposition \ref{nested} is that the smallness of $C$ and $D$ implies good pointwise bounds on the solution. We state Proposition \ref{regularity} as a slightly more quantitative variant of Theorem 4.1 in \cite{dongdu}.

\begin{prop}\label{regularity}
Let $u$, $z_0$, and $\rho$ be as in Proposition \ref{nested} and suppose that for every $z_1\in Q(z_0,\rho_1/2)$, $\rho\in(0,\rho_1/2)$ we have
\eqn{
C(\rho,z_1)+D(\rho,z_1)\leq\epsilon\leq A_1^{-1}.
}
Then, for $j=0,1,2$,
\eqn{
\|\grad^ju\|_{L_{t,x}^\infty(Q(z_0,\rho_1/4))}&\leq A^{O(1)}\epsilon^{1/O(1)}\rho_1^{-1-j}.
}
\end{prop}

\begin{proof}
Let us normalize $\rho_1=1$ and $z_0=0$. By the argument in the proof of Theorem 4.1 in \cite{dongdu} using the bound on $p$ coming from \eqref{critical}, one finds
\eq{\label{linfty}
\|u\|_{L_{t,x}^\infty(Q(1/3))}&\leq A^{O(1)}\epsilon^{1/O(1)}.
}
We may bootstrap the estimates for higher derivatives using Duhamel's formula. Let us fix a decreasing sequence of $O(1)$-many lengths $\frac13>r_1>r_2>r_3>\cdots>\frac14$ satisfying $r_n-r_{n+1}=1/O(1)$. For a frequency $N\gg1$ to be specified, \eqref{linfty}, \eqref{ulin}, Lemma \ref{localbernstein}, and Duhamel's formula for \eqref{ns} starting from $t=-1/3$ imply
\eqn{
\|P_Nu\|_{L_{t,x}^\infty(Q(r_1))}&\lesssim e^{-N^2/O(1)}NA+N^{-1}A^{O(1)}\epsilon^{1/O(1)}+N^{-50}A^2.
}
Clearly with $N$ large enough, the first term (from the linear propagator) is negligible compared to the third (the global contribution to Bernstein). Therefore, again by Duhamel's formula, \eqref{linfty}, and a paraproduct decomposition of $P_N(u\otimes u)$,
\eqn{
\|P_Nu\|_{L_{t,x}
^\infty(Q(r_2))}&\lesssim e^{-N^2/O(1)}NA+N^{-1}\|P_{\lesssim N}u\odot P_{\sim N}u\|_{L^\infty_{tx}(Q(r_1))}\\
&\quad+N^{-1}\sum_{N'\gtrsim N}\|P_{N'}u\|_{L^\infty_{tx}(Q(r_1))}^2+N^{-49}A^2\\
&\lesssim N^{-1}(A^{O(1)}\epsilon^{1/O(1)}+N^{-49}A)(N^{-1}A^{O(1)}\epsilon^{1/O(1)}+N^{-50}A^2)\\
&\quad+N^{-1}\sum_{N'\gtrsim N}((N')^{-1}A^{O(1)}\epsilon^{1/O(1)}+(N')^{-50}A^2)+N^{-50}NA^2\\
&\lesssim N^{-2}A^{O(1)}\epsilon^{1/O(1)}+N^{-49}A^{O(1)}.
}
Thus, once again by Duhamel's formula and \eqref{heat}, for any $N_0>0$,
\eqn{
\|\grad u\|_{L_{t,x}^\infty(Q(r_3))}&\lesssim N_0\|u\|_{L_{t,x}
^\infty(Q(1/3))}+N_0^{-48}A+\sum_{N>N_0}(N\|P_Nu\|_{L_{t,x}^\infty(Q(r_2))}+N^{-48}A)\\
&\leq N_0A^{O(1)}\epsilon^{1/O(1)}+N_0^{-48}A^{O(1)}.
}
By taking $N_0$ to be a suitable power of $\epsilon^{-1}$, we arrive at
\eqn{
\|\grad u\|_{L_{t,x}^\infty(Q(r_3))}&\leq A^{O(1)}\epsilon^{1/O(1)}.
}
Proceeding in the same way, one can obtain the higher order estimates as well.
\end{proof}

Taking Propositions \ref{nested} and \ref{regularity} together, we obtain the useful fact that if $C(Q)+D(Q)\leq A_1^{-1}$, then we have good pointwise bounds for $u$ in $Q/2$. (Clearly we may also replace $Q/2$ with, say, $9Q/10$ by trivially modifying the proofs.) As an application, we prove the first partial regularity result. As discussed in more depth in Section \ref{carleman}, by letting the region of regularity expand in space (as opposed to taking, say, $Q_0\times\mathbb R^{d-k}$ for some small $Q_0\subset\mathbb R^k$), we obtain better estimates upon iterating unique continuation. We remark that we do not claim this to be the optimal result; indeed one should expect that regular regions exist that are unconstrained in up to three of the $d+2$ parabolic dimensions, (cf.\ epochs of regularity when $d=3$ which are unbounded in all three spatial dimensions). In this case, the region is unbounded in only one spatial dimension, i.e., radially toward $\theta$.

\begin{prop}[Slices of regularity]\label{slices}
Assume $u$ is smooth and satisfies \eqref{ns} and \eqref{critical} on $[-T,0]$, $z_0\in[-T/2,0]\times\Rd$, and $R^2\leq T/4$. Then there exist a direction $\theta\in S^{d-1}$ and a time interval $I\subset[t_0-R^2,t_0]$ with $|I|=A_2^{-2}R^2$ such that within the slice
\eqn{
S=I\times\{x\in\Rd:\dist(x,x_0+\mathbb R_+\theta)\leq 10A_2^{-1}|(x-x_0)\cdot\theta|,\,|x-x_0|\geq20R\}\subset[-T,0]\times\Rd,
}
for $j=0,1,2$, we have
\eqn{
\|\grad^ju\|_{L_{t,x}^\infty(S)}\leq A_1^{-1}\bigg(\frac R{A_2}\bigg)^{-1-j}.
}
\end{prop}

\begin{proof}
We normalize $R=1$ and $z_0=0$, then apply Proposition \ref{decomposition} on the interval $[-2,0]$. Let $\mathcal S_0$ be the collection of all spacetime regions of the form
\eqn{
I\times\{x\in\Rd:\dist(x,x_0+\mathbb R_+\theta)\leq20A_2^{-1}|(x-x_0)\cdot\theta|,\,|x-x_0|\geq10\}
}
ranging over all $\theta\in S^{d-1}$ and $I=[-10A_2^{-2}k,-10A_2^{-2}(k-1)]$ where $k\in[1,A_2^2/10]\cap\mathbb N$. Clearly we may find a disjoint subcollection $\mathcal S_1$ containing $\gtrsim A_2^{d+1}$ such slices. We seek to find one where we can apply Propositions \ref{nested} and \ref{regularity}. To find a region where $D$ is small, observe that by the Calder\'on-Zydmund estimate for $\div\div/\Delta$, H\"older's inequality, Sobolev embedding, and \eqref{sharpenergy},
\eqn{
\|\Delta^{-1}\div\div\unlin\otimes\unlin\|_{L_t^1L_x^{\frac d{d-2}}([-1,0]\times\mathbb R^d)}&\lesssim\|\unlin\|_{L_t^2L_x^{\frac{2d}{d-2}}([-1,0]\times\mathbb R^d)}^2\leq A^{O(1)}.
}
By interpolation with the $L_t^\infty L_x^\frac d2$ bound from \eqref{unlin},
\eqn{
\|\Delta^{-1}\div\div\unlin\otimes\unlin\|_{L_{t,x}^2([-1,0]\times\Rd)}&\leq A^{O(1)}.
}
As a result, of the $\gtrsim A_2^{d+1}$ slices in $\mathcal S_1$, at least 99\% must have
\eqn{
\|\Delta^{-1}\div\div\unlin\otimes\unlin\|_{L_{t,x}^2(S)}\leq A_1A_2^{-\frac{d+1}2}.
}
Using \eqref{ulin} and H\"older's inequality, it is easy to see that the same can be said for $\ulin\odot\unlin$ and $\ulin\otimes\ulin$. Let $\mathcal S_2\subset\mathcal S_1$ be the collection of all such slices. Combining these estimates and applying H\"older's inequality, we have
\eq{\label{dbound}
D(Q)\leq A_1A_2^{-\frac34}
}
for every parabolic cylinder $Q\subset S$ of length $\sim A_2^{-1}$ and every $S\in\mathcal S_2$. By the same argument along with \eqref{sharpenergy}, most of the $S\in\mathcal S_2$ satisfy
\eq{\label{goodenergy}
\|\grad \unlin\|_{L_{t,x}^2(S)}\leq A_1A_2^{-\frac{d+1}2},
}
so in fact the family $\mathcal S_3$ of slices satisfying both \eqref{dbound} and \eqref{goodenergy} has $\#(\mathcal S_3)\geq C_0^{-1}A_2^{d+1}$. Each of these slices occupies one of $\sim A_2^{2}$ time intervals, so by the pigeonhole principle, there is an interval $I=[t_0,t_0+A_2^{-2}]$ which contains at least $C_0^{-2}A_2^{d-1}$ slices in $\mathcal S_3$. By \eqref{critical}, there must be one of these slices $S_0$ such that
\eqn{
\|u(t_0)\|_{L_x^d(S_{0,x})}\lesssim A_2^{-1+\frac1d}A
}
where $S_{0,x}\subset\Rd$ is the projection of $S_0$ to the spatial components. Then by H\"older's inequality, for every ball of length $A_2^{-1}$ inside $S_{0,x}$,
\eqn{
\|u(t_0)\|_{L_x^2(B)}\lesssim A_2^{-\frac d2+\frac1d}A.
}
By \eqref{energychange}, \eqref{dbound}, and \eqref{goodenergy}, for any $Q\subset S_0$ of length $A_2/2$,
\eqn{
\|u\|_{L_t^\infty L_x^2(Q/2)}&\leq A_2^{-\frac d2+\frac58}A_1.
}
Note that the bound \eqref{goodenergy} on $\unlin$ can be restricted to any such $Q\subset S_0$ and extended to the full solution $u$ using \eqref{ulin} and H\"older's inequality. We conclude from the above and the local Gagliardo-Nirenberg inequality (see eg.\ Lemma 2.1 in \cite{dongdu}) that
\eqn{
C(Q)&\lesssim A_2^{\frac d2-1}\left(\|\grad u\|_{L_{t,x}^2(Q)}^\frac d{d+3}\|u\|_{L_t^\infty L_x^2(Q)}^{\frac3{d+3}}+\|u\|_{L_t^\infty L_x^2(Q)}\right)\leq A_1A_2^{-\frac38}
}
for any $Q\subset S_0$ of length $A_2/2$. This along with \eqref{dbound} leads to the claimed bounds by Propositions \ref{nested} and \ref{regularity}.
\end{proof}

The next proposition should be compared to Proposition 3.1(vi) in \cite{tao}. In the case $d\geq4$ it will be necessary locate even wider annuli where the solution enjoys good subcritical bounds, at the expense of needing to search a larger range of length scales. Note that in \cite{tao} a key ingredient of the proof is the bounded total speed property which is unavailable in high dimensions. For this reason we proceed in the manner of Barker and Prange who use an $\epsilon$-regularity criterion to find quantitative annuli of regularity; see \cite[Section 6]{barkerprange}.

\begin{prop}[Annuli of regularity]\label{annuli}
Let $u$ be a smooth solution of \eqref{ns} satisfying \eqref{critical} on $[-10,0]$. For any $R_0\geq 2$, there exists a scale $R\in[R_0,R_0^{\exp(A_4)}]$ such that for $j=0,1,2$,
\eqn{
\|\grad^ju\|_{L^\infty_{t,x}([-1,0]\times\{R\leq|x|\leq R^{2A_4}\})}&\leq A_4^{-1/O(1)}.
}
\end{prop}

\begin{proof}
Since, by \eqref{critical},
\eqn{
\int_{[-10,0]\times\{R_0\leq|x|\leq R_0^{\exp A_4}\}}(|u|^d+|p|^{d/2})dxdt\leq A^{O(1)},
}
the pigeonhole principle implies that there exists $R$ in the desired range such that
\eqn{
\int_{[-10,0]\times\{R/10\leq|x|\leq10R^{2A_4}\}}(|u|^d+|p|^{d/2})dxdt\leq A_4^{-\frac12}
}
and therefore, by H\"older's inequality, for every parabolic cylinder $Q\subset[-10,0]\times\{R/10\leq|x|\leq10R^{2A_4}\}$, 
\eqn{
\|u\|_{L_{t,x}^{2+\frac4{d+1}}(Q)}+\|p\|_{L_{t,x}^{1+\frac2{d+1}}(Q)}&\lesssim A_4^{-\frac12}.
}
This implies that the region $[-1,0]\times\{R\leq|x|\leq R^{2A_4}\}$ can be covered by a collection of cylinders $Q_j/2$ such that $Q_j\subset[-10,0]\times\{R/10\leq|x|\leq10R^{2A_4}\}$ and $C(Q_j)+D(Q_j)\lesssim A_4^{-\frac12}$. Successively applying Propositions \ref{nested} and \ref{regularity} in all the $Q_j$ yields the desired bounds.
\end{proof}

\subsection{Backward propagation of concentrations}

Next we prove a high-dimensional analogue of Proposition 3.1(v) in \cite{tao}. The proof given there is obtained by iterating a lemma for very short back-propagation, with the bounded total speed property (Proposition 3.1(ii) in \cite{tao}) preventing the sequence of concentrations from traveling too far through space. Although the bounded total speed is unlikely to hold when $d\geq4$, Proposition \ref{nested} is a suitable replacement.

\begin{prop}\label{bp}
Suppose $u$ is smooth and satisfies \eqref{ns} and \eqref{critical} on $[-T,0]$ where $T\geq100$. If $N_0\geq10A_1$ and
\eqn{
|P_{N_0}u(0)|\geq A_1^{-1}N_0,
}
then there exist $z_1\in[-1,-A_2^{-1}]\times B(A_2)$ and $N_1\in[A_2^{-1},A_2]$ such that
\eqn{
|P_{N_1}u(z_1)|\geq A_2^{-1}.
}
\end{prop}

\begin{proof}
Using Lemma \ref{frequencylocalized} to deduce that there must be a parabolic cylinder about $z=0$ where we still have the lower bound on $|P_{N_0}u|$, we have
\eqn{
A_1^{-1}N_0r^{\frac{(d+1)(d+2)}{2(d+3)}}&\leq\|P_{N_0}u\|_{L_{t,x}^{2+\frac{4}{d+1}}(Q(r))}\lesssim\|u\|_{L_{t,x}^{2+\frac4{d+1}}(Q(A_1^3r))}+A_1^{-50}r^{\frac{(d+1)(d+2)}{2(d+3)}-1}A
}
with $r= A_1^{-2}N_0^{-1}$, using Lemma \ref{localbernstein}. Rearranging, this implies
\eq{
C(A_1N_0^{-1},0)\geq A_1^{-3d}.
}
Because $N_0\geq10A_1$, we can apply Proposition \ref{nested} in the contrapositive to find
\eqn{
C(1,0)+D(1,0)\geq A_1^{-4d}.
}
Suppose first that $C(1,0)\geq\frac12A_1^{-4d}$. Using some large parameter $M$ to be specified, we split $u$ into three pieces to estimate $C(0,1)$: low frequencies
\eqn{
\|P_{<M^{-1}}u\|_{L_{t,x}^{2+\frac4{d+1}}(Q(1))}&\lesssim\|P_{<M^{-1}}u\|_{L_{t,x}^\infty([-1,0]\times\Rd)}\lesssim\frac A{M},
}
intermediate frequencies
\eqn{
\|P_{M^{-1}\leq \,\cdot\,\leq M}u\|_{L_{t,x}^{2+\frac4{d+1}}(Q(1))}&\lesssim\log(M)\max_{M^{-1}\leq N\leq M}\|P_Nu\|_{L_{t,x}^{2+\frac4{d+1}}(Q(1))},
}
and high frequencies
\eqn{
\|P_{>M}u\|_{L_{t,x}^{2+\frac4{d+1}}(Q(1))}&\lesssim\sum_{N>M}\bigg(\|P_N\ulin\|_{L_{t,x}^\infty([-1,0]\times\Rd))}\\
&\quad+N^{\frac d{d+3}}\|P_N\unlin\|_{L_{t,x}^2([-1,0]\times\Rd)}^{1-\frac2{d+3}}\|P_N\unlin\|_{L_t^\infty L_x^2([-1,0]\times\Rd)}^{\frac2{d+3}}\bigg).
}
Here we have used the decomposition from Proposition \ref{decomposition} on, say, $[-2,0]\times\Rd$ followed by Lemma \ref{localbernstein} and H\"older's inequality in space and interpolation in time. For the first term, by \eqref{PNuflat},
\eqn{
\sum_{N>M}\|P_N\ulin\|_{L_{t,x}^\infty([-1,0]\times\Rd)}&\lesssim M^{-50}A^{O(1)}.
}
For the second, by H\"older's inequality, Plancherel, \eqref{unlin}, and \eqref{sharpenergy},
\eqn{
&\sum_{N>M}N^{-\frac1{d+3}}(N\|P_N\unlin\|_{L_{t,x}^2([-1,0]\times\Rd)})^{1-\frac2{d+3}}\|P_N\unlin\|_{L_t^\infty L_x^2([-1,0]\times\Rd)}^{\frac2{d+3}}\\
&\quad\lesssim M^{-\frac1{d+3}}\left(\sum_NN^2\|P_N\unlin\|_{L_{t,x}^2([-1,0]\times\Rd)}^2\right)^{\frac12-\frac1{d+3}}\left(\sum_N\|P_N\unlin\|_{L_t^\infty L_x^2([-1,0]\times\Rd)}^2\right)^{\frac1{d+3}}\nonumber\\
&\quad\leq M^{-\frac1{d+3}}A^{O(1)}.
}
Combining the above estimates, we conclude
\eqn{
\frac12A_1^{-4d}\leq C(1,0)\lesssim\frac AM+\log(M)\max_{M^{-1}\leq N\leq M}\|P_Nu\|_{L_{t,x}^\infty(Q(1))}+M^{-\frac1{d+3}}A^{O(1)}.
}
With $M=A_1^{O(1)}$, we obtain $z_1\in Q(1)$ and $N_1\in[A_2^{-1},A_2]$ such that
\eqn{
|P_{N_1}u(z_1)|\geq A_1^{-O(1)}.
}
Suppose instead that $D(1,0)\geq \frac12A_1^{-4d}$. By H\"older's inequality, Lemma \ref{localbernstein}, and \eqref{critical}, also using the fact that $p=-\Delta^{-1}\div\div(u\otimes u)$, we have
\eqn{
\|P_{<10M^{-1}}p\|_{L_{t,x}^{1+\frac2{d+1}}(Q(1))}&\lesssim \|P_{<10M^{-1}}p\|_{L_{t,x}^\infty(\Rd)}\lesssim M^{-2}A^2.
}
To handle the intermediate and high frequencies, we use the paraproduct decomposition
\eqn{
P_{\geq 10M^{-1}}(u\otimes u)&=P_{\geq 10M^{-1}}\big(2(P_{<M^{-1}}u)\odot (P_{M^{-1}\leq\,\cdot\,\leq M}u)+2u\odot(P_{>M}u)+(P_{M^{-1}\leq\,\cdot\,\leq M}u)^{\otimes 2}\big)\\
&=\Pi_1+\Pi_2+\Pi_3.
}
For the first term, by H\"older's inequality, Lemma \ref{bernstein}, and \eqref{critical},
\eqn{
\|\Delta^{-1}\div\div\Pi_1\|_{L_{t,x}^{1+\frac2{d+1}}(Q(1))}&\lesssim\|P_{<M^{-1}}u\|_{L_{t,x}^\infty([-1,0]\times\Rd)}\|P_{M^{-1}\leq\,\cdot\,\leq M}u\|_{L_t^\infty L_x^d([-1,0]\times\Rd)}\\
&\lesssim A^2M^{-1}.
}
Next, by Proposition \ref{decomposition}, H\"older's inequality, \eqref{critical}, \eqref{PNuflat}, and estimating $P_{>M}\unlin$ using Plancherel and \eqref{sharpenergy} as above, we have
\eqn{
\|\Delta^{-1}\div\div\Pi_2\|_{L_{t,x}^{1+\frac2{d+1}}(Q(1))}&\lesssim A\Big(\|P_{>M}\ulin\|_{L_{t,x}^\infty([-1,0]\times\Rd)}+\|P_{>M}\unlin\|_{L_{t,x}^{2+\frac 4{d+1}}([-1,0]\times\Rd)}\Big)\\
&\lesssim M^{-\frac1{d+3}}A^{O(1)}.
}
Finally, by H\"older's inequality, Lemma \ref{localbernstein}, and \eqref{critical},
\eqn{
\|\Delta^{-1}\div\div\Pi_3\|_{L_{t,x}^{1+\frac2{d+1}}(Q(1))}&\lesssim\|P_{M^{-1}\leq\,\cdot\,\leq M}u\|_{L_{t,x}^{2+\frac4{d+1}}([-1,0]\times B(M^2))}^2+M^{-50}A^2.
}
In total,
\eqn{
\frac12A_1^{-4d}\leq D(0,1)&\lesssim M^{-\frac12}A+M^{-\frac1{2(d+3)}}A^{O(1)}+\log(M)\max_{N\in[M^{-1},M]}\|P_{N}u\|_{L_{t,x}^\infty(Q[-1,0]\times B(M^2))}.
}
Once again with $M=A_1^{O(1)}$, we obtain $N_1$ and $z_1$ with the claimed properties. Finally we address the possibility that this $t_1$ falls in $[-A_2^{-1},0]$ instead of the desired interval. By the fundamental theorem of calculus and Lemma \ref{frequencylocalized},
\eqn{
|P_{N_1}u(t_1-A_2^{-1},x_1)|&\geq A_3^{-O(1)}-O(N_1^3A^2A_2^{-1})
}
which implies we can redefine $t_1$ to be in $[-1,-A_2^{-1}]$ while maintaining the lower bound on $|P_{N_1}u|$.
\end{proof}

\section{Carleman inequalities for unique continuation}\label{carleman}

For the reader's convenience, we begin by quoting the quantitative unique continuation Carleman inequality from \cite{tao}, which has a straightforward generalization to $d$ dimensions. 

\begin{prop}[Unique continuation]\label{uc}
With $T,r>0$, assume $u\in C^\infty([0,T]\times B(r)\to\mathbb R)$ admits the differential inequality
\eq{\label{differentialinequality}
|Lu|\leq\frac{|u|}{C_0T}+\frac{|\grad u|}{(C_0T)^{1/2}},
}
while the parameters satisfy
\eqn{
r^2\geq4000T,\quad0<t_1\leq t_0\leq\frac T{1000d}.
}
Then
\eqn{
&\int_{t_0}^{2t_0}\int_{B(r/2)}(T^{-1}|u|^2+|\grad u|^2)e^{-|x|^2/4t}dxdt\lesssim e^{-\frac{r^2}{1000t_0}}X+t_0^{\frac d2}(3et_0/t_1)^{O(r^2/t_0)}Y
}
where
\eqn{
X&:=\int_0^T\int_{B(r)}(T^{-1}|u|^2+|\grad u|^2)dxdt,\\
Y&:=\int_{B(r)}|u(0,x)|^2t_1^{-\frac d2}e^{-|x|^2/4t_1}dx.
}
\end{prop}

Clearly this Carleman inequality as written is incompatible with the geometry of Proposition \ref{slices} since $B(r)$ would have to be contained in the thin slice in order to guarantee \eqref{differentialinequality}, while simultaneously we need $r^2\gg t_0$ in order for the first error term to be suppressed. Instead we iteratively apply the Carleman inequality outward in space, starting near the vertex of the slice. The point is that as the iteration proceeds, the center for the Carleman inequality moves further in the $\theta$ direction, so $r$ can be taken to be larger, which makes the Carleman inequality stronger. Thus combining Propositions \ref{slices} and \ref{iteratedcarleman} leads to a feedback loop which leads to substantially better estimates; specifically, only $\sim\log (R_2/R_1)$ iterations of Proposition \ref{uc} (by way of Lemma \ref{onestep}) are needed\footnote{If instead one were to iterate the Carleman inequality through a region of the form $Q_0\times\mathbb R^{d-k}$ for some small $Q_0\subset\mathbb R^k$, one would need a number of iterations on the order of $R_2/R_1$. This would lead to an extra exponential in the vorticity lower bound, which would in turn require us to ensure a much smaller error when the backward uniqueness Carleman inequality is applied in the proof of Proposition \ref{prop2}. It would be necessary then to find a much larger annulus of regularity in Proposition \ref{annuli} which would result (rather unsatisfyingly) in tower exponential bounds in Theorem \ref{regularitythm}.} to propagate a concentration from length scale $R_1$ to $R_2$.

\begin{prop}[Iterated unique continuation Carleman inequality]\label{iteratedcarleman}
Suppose $T_1>0$, $0<\eta\leq C_0^{-1}$, and $u$ is smooth on $S$ with
\eq{\label{reg}
\|\grad^ju\|_{L_{t,x}^\infty(S)}\leq (\eta T_1)^{-1-\frac{j}2},\quad|Lu|\leq\frac{|u|}{C_0\eta T_1}+\frac{|\grad u|}{(C_0\eta T_1)^\frac12}\quad\forall(t,x)\in S
}
for $j=0,1$, where, for some direction $\theta\in S^{d-1}$,
\eqn{
S=[-\eta T_1,0]\times\{x\in\Rd:|x|>10T_1^\frac12,\,\dist(x,\mathbb R_+\theta)\leq \eta|x\cdot\theta|\}.
}
Moreover, assume that for every $t\in[-\eta T_1,0]$, we have
\eqn{
\int_{B(R_0\theta,\eta^{5}R_0)}|u(t)|^2dx\geq\epsilon T_1^{\frac d2-2}
}
where $20T_1^\frac12\leq R_0\leq \eta^{-2}T_1^\frac12$ and $\epsilon\leq \eta^8$. Then for every $t\in[-
\eta T_1/2,0]$ and $R\geq 2R_0$, we have
\eqn{
\int_{B(R\theta,\eta^{5}R)}|u(t)|^2dx\geq \epsilon^{(R/R_0)^{\eta^{-4}}}T_1^{\frac d2-2}.
}

\end{prop}

Given the following lemma, Proposition \ref{iteratedcarleman} will follow by iteration.

\begin{lem}\label{onestep}
Assume $u$, $T_1$, and $\eta$ are as in Proposition \ref{iteratedcarleman} and that there is some $R\geq20T_1^\frac12$ and $a\in(\frac12,1)$ such that for every $t\in[-aT_1,0]$,
\eqn{
\int_{B(R\theta,\eta^5R)}|u(t)|^2dx\geq\epsilon_0 T_1^{\frac d2-2}
}
where
\eqn{
\epsilon_0\leq\min(\eta^8,(R^2/T_1)^{-50d\eta},e^{-2000d\hspace{.003in}\eta^4R^2/T_1}).
}
Then for every $t\in[-aT_1+2\eta^5R^2\log^{-1}\frac1{\epsilon_0},0]$,
\eqn{
\int_{B(R'\theta,\eta^5R')}|u(t)|^2dx\geq \epsilon_0^{\eta^{-2}} T_1^{\frac d2-2}
}
where $R':=(1+\eta^3)R$.
\end{lem}

\begin{proof}[Proof of Proposition \ref{iteratedcarleman}]
Let us normalize $T_1=1$. One iterates Lemma~\ref{onestep} on the time intervals $[-a_k,0]$ for $k=0,1,\ldots,n$, where $n=\lceil\log_{1+\eta^3}(R/R_0)\rceil$. Specifically, the $k$th application of the lemma is centered at the point $R_k\theta\in\Rd$ and uses the lower bound $\epsilon_k$, where
\eqn{
\epsilon_k=\epsilon^{\eta^{-2k}},\quad R_k=R_0(1+\eta^3)^k,\quad a_k=\eta-2\sum_{i=0}^k\eta^5R_i^2\log^{-1}\frac1{\epsilon_i}.
}
One computes that
\eqn{
a_k&=\eta-2\eta^5R_0^2\log^{-1}\frac1\epsilon\sum_{i=0}^k(\eta+\eta^4)^{2i}\geq\eta-4\eta^{5}R_0^2\log^{-1}\frac1\epsilon.
}
Recall that $R_0\leq\eta^{-2}$ and $\epsilon\leq\eta^8$. Thus, with $\eta$ sufficiently small, $a_k\geq\frac\eta2$ so the claimed bound holds on $[-\frac\eta2,0]$. The final lower bound resulting from the iteration is given by
\eqn{
\epsilon_n=\epsilon^{\eta^{-2\lceil\log_{1+\eta^3}(R/R_0)\rceil}}\geq\epsilon^{\eta^{-2}(R/R_0)^{\frac{\log\eta^{-2}}{\log(1+\eta^3)}}}.
}
With $\eta$ sufficiently small, we have $\frac{\log\eta^{-2}}{\log(1+\eta^3)}\leq\eta^{-\frac72}$ and $\eta^{-2}\leq(R/R_0)^{\eta^{-\frac12}}$, using that $R\geq2R_0$. Thus $\epsilon_n\geq\epsilon^{(R/R_0)^{\eta^{-4}}}$ as claimed.
\end{proof}

\begin{proof}[Proof of Lemma \ref{onestep}]
Again, we rescale so that $T_1=1$. Fix any $t'\in[-a+2\eta^5R^2\log^{-1}\frac1{\epsilon_0},0]$. We apply Proposition \ref{uc} to the function
\eqn{
(t,x)\mapsto u(t'-t,x+R'\theta)
}
on the time interval $[0,T_c]$ with the parameters
\eqn{
T_c=\min(\eta/2,\eta^5R^2),\quad r=\eta^2R,\quad t_0=\eta^{5}R^2\log^{-1}\frac1{\epsilon_0},\quad t_1=\eta^{15}R^2\log^{-1}\frac1{\epsilon_0}.
}
Clearly \eqref{reg} implies \eqref{differentialinequality} is satisfied. Consider the three terms in the Carleman inequality which takes the form $Z\leq X+Y$. For the left-hand side, since $B(R'\theta,\eta^2R/2)\supset B(R\theta,\eta^5R)$,
\eqn{
Z&\gtrsim t_0T_c^{-1}\int_{B(R'\theta,\eta^2R/2)}|u|^2e^{-|x-R'\theta|^2/4t_0}dx\gtrsim\max(\eta^4R^2,1)\log^{-1}\Big(\frac1{\epsilon_0}\Big)\epsilon_0^{1+\eta^5/4}\geq \epsilon_0^2,
}
using that $\eta$ and $\epsilon_0$ are small, $R\geq20$, and $\eta\geq\epsilon_0^\frac18$. Next, by \eqref{reg},
\eqn{
X\leq\epsilon_0^{\eta^{-1}/1000}\eta^{2+2d}R^d
}
which is negligible compared to $Z$ due to the constraint $\eta\leq C_0^{-1}$. For the remaining term in the Carleman inequality,
\eqn{
Y\leq\epsilon_0^{-\eta^{-2}}\int_{|x-R'\theta|\leq \eta^2R}|u(t',x)|^2e^{-|x-R'\theta|^2/4t_1}dx.
}
By \eqref{reg}, the contribution to this term from the region where $|x-R'\theta|>\eta^5R'$ is negligible compared to $Z$:
\eqn{
\epsilon_0^{-\eta^{-2}}\int_{\eta^6R'<|x-R'\theta|\leq\eta^2R}|u(t',x)|^2e^{-|x-R'\theta|^2/4t_1}dx&\lesssim\eta^{2d-2}R^d\epsilon_0^{\eta^{-3}/4-\eta^{-2}}\ll\epsilon_0^2,
}
using that $\epsilon_0\leq R^{-100d\eta}$. Thus $Z$ is bounded by the contribution to $Y$ from $B(R'\theta,\eta^6R')$ which proves the lemma.
\end{proof}

\section{Main propositions}\label{mainprop}

Next we proceed to the main propositions of this paper. The philosophy is similar to \cite{tao} but as discussed in Sections \ref{introduction} and \ref{tools}, in higher dimensions we do not have quantitative epochs of regularity. As a result, given a spacetime point where $u$ has a high frequency concentration, it is far from clear that the vorticity lower bound implied by Proposition \ref{bp} intersects at all with a spacetime region where the solution is regular, let alone an entire epoch $I\times\Rd$ as in the three-dimensional case. From a qualitative perspective, since $\omega$ is locally $L_{t,x}^2$ and the measure of the spacetime set where $|\grad^ju|\lesssim\ell^{-1-j}$ shrinks to zero as $\ell\to0$, there must be some small $\ell$ and cylinders $Q'\subset\subset Q$ of length $\sim\ell$ such that $|\grad^ju|\lesssim \ell^{-1-j}$ holds in $Q$ while $\int_{Q'}|\omega|^2dxdt$ is bounded from below (see Figure \ref{fig1}). The problem is that in order to prove a quantitative theorem, we need an effective lower bound on this $\ell$.

As one sees in the proof of Proposition \ref{prop1}, the worst-case scenario is that at each small scale $\ell$, there are $\sim\ell^{-d+2}$ parabolic cylinders of length $\sim\ell$ where $\int|\omega|^2dxdt\gtrsim \ell^{d-2}$, and in the complement the solution obeys $|\grad^ju|\lesssim\ell^{-1-j}$. At each scale, this fractal configuration is consistent with the energy inequality. We rule out this scenario in dimensions $d\geq5$ by applying the improved energy bound \eqref{maximalregularity} at a sufficiently small scale $\ell=A^{-C}$. In $d=4$ we cannot quite use this improvement and are forced to take $\ell$ as small as $\exp(-A^C)$. Here the idea is that each scale $\ell$ contributes roughly a fixed amount to the energy. A significant fraction of the contribution comes from the frequencies around $\ell^{-1}$, so by summing over many scales we can contradict this scenario.

Note that the exponential smallness of $\ell$ when $d=4$ does not affect the final estimates because it contributes in parallel with exponentials appearing at other points in the argument.\footnote{It is conceivable that the $d\geq5$ case can be handled using the same energy pigeonholing approach, although it is less straightforward because of the spatial overlaps of the concentrations caused by the fact that $Q'$ is a factor $\delta$ smaller than $Q$. As a result $\ell$ would depend exponentially on $\delta^{-1}$ which would cause problems in the proof of Proposition \ref{prop2}, as the smallness of $\delta$ is necessary to create favorable geometry for the Carleman estimates. It is preferable for other reasons to have $\ell$ depend polynomially on $A$; for example see Remark \ref{remark}.}

\begin{prop}[Backward propagation into a regular region]\label{prop1}
Suppose $u$ is a classical solution of \eqref{ns} on $[t_0-T,t_0]$ satisfying \eqref{critical}, and that there are $x_0\in\Rd$ and $N_0>0$ such that at the point $z_0=(t_0,x_0)$,
\eqn{
|P_{N_0}u(z_0)|\geq A_1^{-1}N_0.
}
Then for any $T_1\in[A_1^2N_0^{-2},T/100]$, there exist $\ell>0$ and $Q=Q(z_0',\ell/2)\subset[-T_1,-A_2^{-1}T_1]\times B(A_3T_1^\frac12)$ such that
\eq{\label{finalbounds}
\|\grad^ju\|_{L_{t,x}^\infty(Q)}\leq A_2^{-1}\ell^{-j-1}
}
for $j=0,1,2$ and
\eq{
\|\omega\|_{L_{t,x}^2(Q')}\geq A_3^{-O(1)}(\delta\ell)^{\frac d2+1}T_1^{-1}\label{finalconc}
}
where $Q'=Q(z_0'-(\ell^2/8,0),\delta\ell)$. We may take $\delta=A_4^{-1}$ and $\ell$ such that
\begin{equation}\begin{aligned}\label{ell}
\ell&\in[\exp(-A_4),A_4^{-1}],&d=4,\\\ell&=A_4^{-2d-1},&d\geq5.
\end{aligned}\end{equation}
\end{prop}

\begin{proof}

\begin{figure}
    \centering
    \begin{tikzpicture}[scale=1.25,every text node part/.style={align=center}]
    \draw[line width=.22mm](0,-5)--(0,0);
    \draw[line width=.22mm](-5,0)--(5,0);
    \draw(-.08,-4.7)--node[left]{$t_0-T_1$}(.08,-4.7);
    \fill[red,opacity=.8](-.2,-.4)rectangle(.2,0);
    \filldraw[black,opacity=1](0,0)circle(1.2pt)node[anchor=south]{$z_0$};
    \fill[red,opacity=.8](.9,-3.7)rectangle(3.5,-1.3);
    \draw[->](.3,-.3)--(1.9,-1.15);
    \node at(1.6,-.5){Prop.\ \ref{bp}};
    \node at(2.2,-4.15){$Q(z_1,A_2^{-4}T_1^{\frac12})$\\($|P_{N_1}u|$ is bounded below)};
    \filldraw[black,opacity=1](2.2,-1.3)circle(1.2pt)node[anchor=south]{$z_1$};
    \fill[blue!50!white,opacity=.6](-4.2,-3)rectangle(4.2,-2.1);
    \node at(-1.2,-2.56){$I\times B(A_3T_1^\frac12)$\\($\omega$ concentration)};
    \draw[line width=.27mm](-3.29,-2.32-.6)rectangle(-2.81,-1.87-.6);
    \draw[|-|](-2.81,-2.49-.6)--(-3.29,-2.49-.6);
    \node at(-3.05,-2.63-.6){\small$\ell$};
    \fill[black](-3.05-.08,-2.18-.6)rectangle(-3.05+.08,-2.01-.6);
    \draw[|-|](-3.4,-2.18-.6)--(-3.4,-2.01-.6);
    \node at(-3.65,-2.7){\small$\delta\ell$};
    \node at(-3.05,-2.3){\small$Q'\subset Q$};
    \end{tikzpicture}
    \caption{We schematize some key steps in the proof of Proposition \ref{prop1}. The high frequency concentration at $z_0$ is propagated backward in time to $z_1$. The concentration of $P_{N_1}u$ persists in a parabolic cylinder (red) which we convert into a lower bound on $\|\omega\|_{L_{t,x}^2}$ (blue). The objective is to locate a small cylinder $Q$ such that $u$ obeys subcritical bounds in the interior and the vorticity concentrates on a smaller subcylinder $Q'$.}
    \label{fig1}

\end{figure}
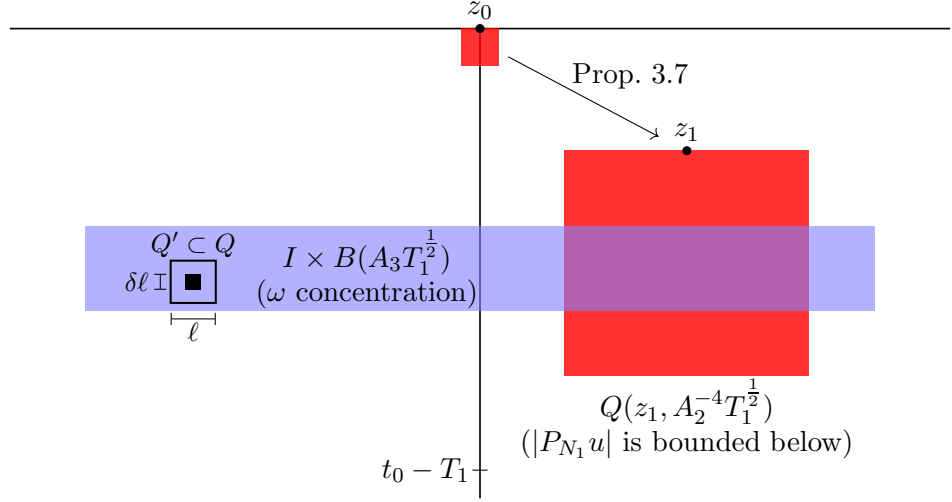

Without loss of generality we may let $z_0=0$ and $T_1=1$. Let us begin with the case $d\geq5$. By Proposition \ref{bp}, there exists a point $z_1\in[-1,-A_2^{-1}]\times B(A_2)$ and a frequency $N_1\in[A_2^{-1},A_2]$ such that
\eqn{
|P_{N_1}u(z_1)|\geq A_2^{-1}.
}
Combining this with Lemma \ref{frequencylocalized}, we find that the lower bound persists in a parabolic cylinder:
\eq{\label{lb}
|P_{N_1}u(z)|\gtrsim A_2^{-1},\quad \forall z\in Q(z_1,A_2^{-4}).
}
We apply Proposition \ref{decomposition} on $[t_1-2A_2^{-4},t_1]$ to obtain a decomposition $u=\ulin+\unlin$. Let $I$ be the contraction of the time interval $[t_1-A_2^{-4},t_1]$ by a factor of $\frac15$ about its center. By \eqref{maximalregularity}, H\"older's inequality, and \eqref{ulin},
\eq{\label{supercriticalregularityeq}
\|\grad u\|_{L_{t,x}^\frac d2(5I\times B(2A_3))}\leq A_3^{O(1)}.
}
Defining the vorticity $\omega:=dv$ where $d$ is the exterior derivative on $\Rd$ and $v$ is the covelocity field of $u$, we apply the codifferential $\delta$ to obtain $-\Delta  v=\delta\omega$. Thus we have a version of the Biot-Savart law,
\eqn{
v=-\Delta^{-1}\delta \omega.
}
It follows from \eqref{lb} and Lemma \ref{localbernstein} that for all $t\in I$,
\eqn{
A_2^{-O(1)}&\lesssim\|P_{N_1}\Delta^{-1}\delta\omega(t)\|_{L_x^2(B(A_3/2))}\\
&\lesssim N_1^{-1}\|\omega(t)\|_{L_{x}^2(B(A_3))}+(A_3N_1)^{-50d}N_1^{-1}(\|\grad\unlin\|_{L_x^2(\mathbb R^d)}+A_3^{\frac d2}\|\grad\ulin\|_{L_x^\infty(\Rd)}).
}
Taking the $L_t^2(I)$ norm, bounding the $\unlin$ global error term with \eqref{sharpenergy}, and the $\ulin$ term with \eqref{ulin}, we obtain
\eq{\label{startingconcentration}
\|\omega\|_{L_{t,x}^2(I\times B(A_3))}\geq A_2^{-O(1)}.
}
Consider the collection of parabolic cylinders
\eqn{
\mathcal C_0:=\{Q(z,\ell):z\in((\delta\ell)^2\mathbb Z\times(\delta\ell\hspace{.01in}\mathbb Z)^{d})\cap (2I\times B(2A_3))\}
}
of which there are $\sim A_2^{-5}A_3^{d}(\delta\ell)^{-d-2}$. (Once again $2I$ denotes dilation of the interval about its center.) We seek to understand in which cylinders $u$ is regular. By the $L^p$-boundedness of $\div\div/\Delta$, H\"older's inequality, Sobolev embedding, and \eqref{sharpenergy},
\eqn{
\|\Delta^{-1}\div\div\unlin\otimes\unlin\|_{L_t^1L_x^{\frac d{d-2}}([t_1-A_2^{-4},t_1]\times\mathbb R^d)}&\lesssim\|\unlin\|_{L_t^2L_x^{\frac{2d}{d-2}}([t_1-A_2^{-4},t_1]\times\mathbb R^d)}^2\leq A_2^{O(1)}.
}
By interpolation with the $L_t^\infty L_x^\frac d2$ bound coming from \eqref{unlin},
\eqn{
\|\Delta^{-1}\div\div\unlin\otimes\unlin\|_{L_{t,x}^2([t_1-A_2^{-4},t_1]\times\Rd)}&\leq A_2^{O(1)}.
}
Using this, \eqref{critical}, and \eqref{sharpenergy},
\eq{\label{total}
&
\sum_{Q\in\mathcal C_0}\Big(\|\grad\unlin\|_{L_{t,x}^2(Q)}^2+\|\Delta^{-1}\div\div\unlin\otimes\unlin\|_{L_{t,x}^2(Q)}^2+\|u\|_{L_{t,x}^d}^d\Big)\leq \delta^{-d-2}A_2^{O(1)},
}
since the sets in $\mathcal C_0$ can overlap up to $O(\delta^{-d-2})$ times. Define
\eqn{
\mathcal C_1:=\Big\{Q\in\mathcal C_0:\max\big(\|\grad\unlin\|_{L_{t,x}^2(Q)},\|\Delta^{-1}\div\div\unlin\otimes\unlin\|_{L_{t,x}^2(Q)},\|u\|_{L_{t,x}^d(Q)}^{d/2}\big)>A_3^{-1}\ell^{\frac d2-1}\Big\}.
}
From \eqref{total}, we clearly have
\eqn{
\#(\mathcal C_1)\leq \delta^{-d-2}\ell^{2-d}A_3^2A_2^{O(1)}\leq\frac1{100}\#(\mathcal C_0).
}
Consider an arbitrary $Q_0=I_0\times B_0\in\mathcal C_0\setminus\mathcal C_1$. Additionally using \eqref{ulin} and \eqref{unlin}, we have
\eqn{
\|p\|_{L_{t,x}^2(Q_0)}&\leq A_3^{-1}\ell^{\frac d2-1}+\|\Delta^{-1}\div\div(2\unlin\odot\ulin+\ulin\otimes\ulin)\|_{L_{t,x}^2(Q_0)}\\
&\lesssim A_3^{-1}\ell^{\frac d2-1}+\ell^{\frac d2}A_2^{O(1)}.
}
Then by H\"older's inequality
\eq{\label{Dbound}
D(Q_0)\lesssim A_3^{-1}+\ell A_2^{O(1)}\lesssim A_3^{-1}.
}
Next we address $C(Q_0)$. Let $I_{1/10}$ be the first $\frac1{10}$ of the interval $I_0$. Using again that $Q_0\in\mathcal C_0\setminus\mathcal C_1$,
\eqn{
\int_{I_{1/10}}\|u\|_{L_x^d(B)}^ddt\leq A_3^{-2}\ell^{d-2}
}
and so by the pigeonhole principle and H\"older's inequality, there exists a $\tau_0\in I_{1/10}$ such that
\eqn{
\|u(\tau_0)\|_{L_x^2(B_0)}\lesssim\ell^{\frac d2-1}\|u(\tau_0)\|_{L_x^d(B_0)}\lesssim A_3^{-\frac2d}\ell^{\frac d2-\frac4d}.
}
With this we can apply \eqref{energychange}, \eqref{Dbound}, and the fact that $Q_0\notin\mathcal C_1$ (along with H\"older's inequality and \eqref{ulin} for the $\ulin$ part) to obtain
\eqn{
\|u\|_{L_t^\infty L_x^2(3Q_0/4)}\lesssim A_3^{-\frac2d}\ell^{\frac d2-\frac4d}+A_3^{-\frac12}\ell^{\frac d2-1}A+A_2^{O(1)}\ell^{\frac d2}.
}
A bound for $\|\grad u\|_{L_{t,x}^2(3Q_0/4)}$ similarly follows from the definition of $\mathcal C_1$ and \eqref{ulin}. Then by Gagliardo-Nirenberg interpolation,
\eqn{
C(3Q_0/4)\lesssim A_3^{-\frac2d}A+\ell A_2^{O(1)}.
}
With this and \eqref{Dbound}, we arrive at \eqref{finalbounds} in $Q_0/2$ by Propositions \ref{nested} and \ref{regularity}.

For every $Q=Q(z,\ell)\in\mathcal C_0$, let $Q':=Q(z-(\ell^2/8,0),\delta\ell)$. Since $\{Q':Q\in\mathcal C_0\}$ covers $I\times B(R)$, \eqref{startingconcentration} implies
\eqn{
\sum_{Q\in\mathcal C_0}\|\omega\|_{L_{t,x}^2(Q')}^2\geq 2A_3^{-1}.
}
There are two cases. First, suppose
\eqn{
\sum_{Q\in\mathcal C_0\setminus\mathcal C_1}\|\omega\|_{L_{t,x}^2(Q')}^2\geq A_3^{-1}.
}
By the pigeonhole principle, since the family $\mathcal C_0\setminus\mathcal C_1$ has cardinality $A_3^{O(1)}(\delta\ell)^{-d-2}$, there is a $Q\in\mathcal C_0\setminus\mathcal C_1$ such that
\eq{\label{goodnews}
\|\omega\|_{L_{t,x}^2(Q')}\geq A_3^{-O(1)}(\delta\ell)^{\frac d2+1}.
}
This pair $Q,Q'$ satisfies the conclusion of the proposition. In the other case,
\eq{\label{hardcase}
\sum_{Q\in\mathcal C_1}\|\omega\|_{L_{t,x}^2(Q')}^2\geq A_3^{-1}.
}
If so, we seek to derive a contradiction with \eqref{supercriticalregularityeq}. We compare the lower bound \eqref{hardcase} with
\eqn{
\|\omega\|_{L_{t,x}^2(Q')}^2\leq A^{O(1)}(\delta\ell)^{d-2}
}
from \eqref{firstlocalenergy}, and the fact that $\mathcal C_1$ contains at most $A_3^3\delta^{-d-2}\ell^{-d+2}$ cylinders. Indeed, defining the family of disjoint cylinders
\eqn{
\mathcal C_2:=\{Q':Q\in\mathcal C_1,\,\,\|\omega\|_{L_{t,x}^2(Q')}^2>A_3^{-5}\delta^{d+2}\ell^{d-2}\},
}
we have, using that the contracted cylinders $\{Q'\}_{Q\in\mathcal C_0}$ are disjoint,
\eqn{
A_3^{-1}&\leq\sum_{Q\in\mathcal C_1}\|\omega\|_{L_{t,x}^2(Q')}^2\leq\#(\mathcal C_2)A^{O(1)}(\delta\ell)^{d-2}+\#(\mathcal C_1\setminus\mathcal C_2)A_3^{-5}\delta^{d+2}\ell^{d-2}.
}
It follows that
\eq{\label{C2number}
\#(\mathcal C_2)\geq A_3^{-2}(\delta\ell)^{-d+2}.
}
For all $Q'\in\mathcal C_2$ and $p\geq2$, by H\"older's inequality,
\eqn{
\|\omega\|_{L_{t,x}^\frac d2(Q')}\gtrsim A_3^{-\frac52}\ell^{\frac4d}\delta^{2+\frac4d}.
}
Summing over $\mathcal C_2$,
\eq{\label{vorticitytotal}
\|\omega\|_{L_{t,x}^\frac d2(2I\times B(2A_3))}\geq A_3^{-O(1)}\ell^{\frac 8d-2}\delta^\frac8d.
}
With $\ell$ sufficiently small as in \eqref{ell}, this is in contradiction with \eqref{supercriticalregularityeq}.

Next consider the case $d=4$. We define
\eqn{
\mathcal C_{3}:=\big\{Q\in\mathcal C_0:\|\grad^ju\|_{L_{t,x}^\infty(Q/2)}\leq A_2^{-1}\ell^{-j-1}\text{ for }j=0,1,2\big\}.
}
There are two cases: first, suppose $\bigcup_{Q\in\mathcal C_0\setminus\mathcal C_3}5Q$ projected to the time axis does not cover $I$. Then there exists an interval $I'\subset I$ of length $\ell^2$ such that
\eqn{
\|\grad^ju\|_{L_{t,x}^\infty(I'\times B(A_3))}\leq A_2^{-1}\ell^{-j-1}
}
for $j=0,1,2$. The existence of a large slab of regularity makes this case relatively straightfoward so we argue briefly. One appeals once again to \eqref{lb} and repeats the calculations leading to \eqref{startingconcentration}; however now when we take the $L_t^2$ norm of the Bernstein inequality it is only over $I'$ which yields the lower bound $\|\omega\|_{L_{t,x}^2(I'\times B(A_3))}\geq A_2^{-O(1)}\ell$. Analogous to the definition of $\mathcal C_0$, we partition a slight dilation of $I'\times B(A_3)$ into overlapping parabolic cylinders of length $\ell$ offset by length $\delta\ell$. Using the regularity assumed within $I'$ and applying the pigeonhole principle to the vorticity lower bound, it is clear that there exist $Q$ and $Q'$ obeying \eqref{finalbounds} and \eqref{finalconc}.

Otherwise, suppose $\bigcup_{Q\in\mathcal C_0\setminus\mathcal C_3}5Q$ when projected to the time axis does cover $I$. Then we may take a $\mathcal C_4\subset\mathcal C_0\setminus\mathcal C_3$ such that the projections of $\{5Q\}_{Q\in\mathcal C_4}$ form a subcover which is minimal in the sense that no more than two intersect at once. It follows that
\eqn{
\#(\mathcal C_4)\geq A_2^{-O(1)}\ell^{-2}.
}
Due to our definition of $\mathcal C_3$, for every $Q\in\mathcal C_4$, applying Propositions \ref{nested} and \ref{regularity} in the converse yields
\eqn{
C(Q)+D(Q)>A_2^{-O(1)}.
}
By the argument from the proof of Proposition \ref{bp}, there exist $N\in[A_3^{-1}\ell^{-1},A_3\ell^{-1}]$ and $z\in A_3Q$ such that
\eqn{
|P_Nu(z)|\geq A_3^{-1}\ell^{-1}.
}
It follows by Lemmas \ref{frequencylocalized} and \ref{localbernstein}, as well as H\"older, \eqref{ulin}, and \eqref{sharpenergy} to estimate the global Bernstein error, that
\eqn{
\|P_N\grad u\|_{L_{t,x}^2(A_3^2Q)}\geq A_3^{-O(1)}\ell.
}
Using H\"older's inequality and \eqref{PNuflat}, one computes that the contribution from $\ulin$ is negligible thanks to the smallness of $\ell$. (Note that we continue to refer to the decomposition obtained by applying Proposition \ref{decomposition} on $[t_1-2A_2^{-4},t_1]$.) By the properties of $\mathcal C_4$, particularly the at most $A_3^{O(1)}$-fold boundedness of the overlap, we obtain
\eq{\label{energylb}
\sum_{N\in[A_3^{-1}\ell^{-1},A_3\ell^{-1}]}\|P_N\grad \unlin\|_{L_{t,x}^2(2I\times\Rd)}^2\geq A_3^{-O(1)}.
}
On the other hand, by Plancherel and \eqref{sharpenergy},
\eqn{
\sum_N\|P_N\grad \unlin\|_{L_{t,x}^2(I\times\Rd)}^2\leq A_2^{O(1)}.
}
If \eqref{energylb} holds for all $\ell\in[\exp(-A_4),A_4^{-1}]$, we reach a contradiction by summing over a geometric sequence of scales in this range. Thus the proposition is satisfied by fixing $\ell$ to be any scale for which \eqref{energylb} fails.
\end{proof}

Having obtained a suitable vorticity concentration within a cylinder where the solution is regular, we need only to propagate this lower bound back to time $t_0$ using a series of Carleman inequalities. For every scale $T_1$ between $N_0^{-2}$ and $T$, this scheme leads to a triple-exponentially small amount of $L_x^3$ mass at $t_0$. Summing over $\log(TN_0^2)$-many geometrically separated scales and comparing the result to \eqref{critical}, we will conclude the following.

\begin{prop}[Propagation forward to the final time]\label{prop2}
Suppose $u$, $z_0$, and $N_0$ are as in Proposition \ref{prop1}. Then
\eqn{
TN_0^2\leq\exp\exp\exp\exp(A_6).
}
\end{prop}

\begin{proof}

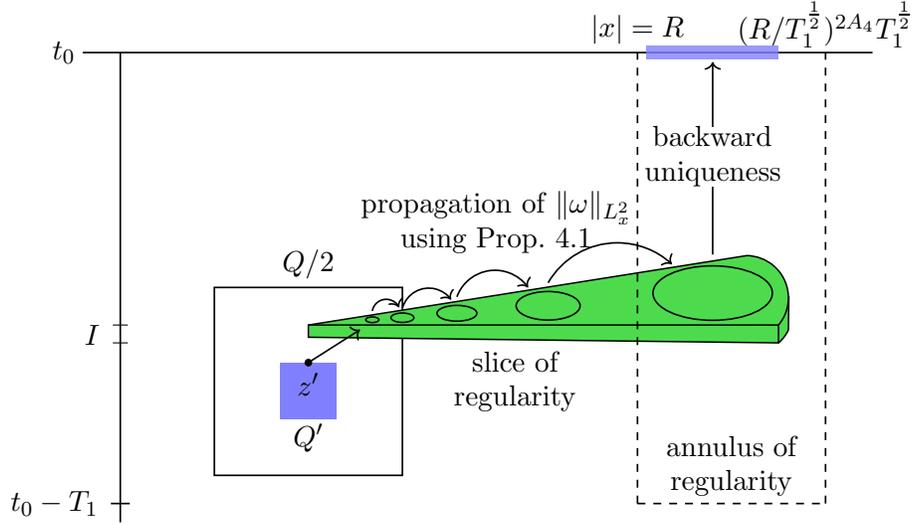
\begin{figure}
    \centering
    \begin{tikzpicture}[scale=1.25,every text node part/.style={align=center}]
    \draw[line width=.22mm](-.4,.5)--(8,.5);
    \node at(-.6,.5){$t_0$};
    \draw[line width=.22mm](0,.5)--(0,-4.5);
    \draw[line width=.22mm](-.1,-4.3)--(.1,-4.3);
    \draw[line width=.22mm](1,-4)rectangle(3,-2);
    \node at(-.7,-4.3){$t_0-T_1$};
    \node at(2,-1.75){$Q/2$};
    \fill[blue!50!white,opacity=1](2-.3,-3.4)rectangle(2+.3,-2.8);
    \filldraw[black,opacity=1](2,-2.8)circle(1.pt)node[anchor=north]{$z'$};
    \node at(2,-3.6){$Q'$};
    \draw(-.08,-2.4)--(.08,-2.4);
    \draw(-.08,-2.59)--(.08,-2.59);
    \node at (-.3,-2.5){$I$};
    \filldraw[fill=black!20!green!70!white,opacity=1,line width=.20mm](7,-2.59)arc(-50:-25:.4)--(7.11,-2.1)--(6,-2);
    \filldraw[fill=black!20!green!70!white,opacity=1,line width=.20mm](7,-2.4)arc(-40:89:.45)--(2,-2.4);
    
    \draw[fill=black!20!green!70!white,opacity=1,line width=.20mm](2,-2.53)--(7,-2.59)--(7,-2.4)--(2,-2.4)--cycle;
    \node at(4.2,-3){slice of\\regularity};
    \draw[line width=.21mm](2.68,-2.344)ellipse(.07 and .03);
    \draw[line width=.21mm](3.,-2.322)ellipse(.122 and .05);
    \draw[line width=.21mm](3.58,-2.275)ellipse(.212 and .084);
    \draw[line width=.21mm](4.55,-2.195)ellipse(.34 and .152);
    \draw[line width=.21mm](6.3,-2.06)ellipse(.635 and .29);
    \draw[->,line width=.22mm](2,-2.8)--(2.55,-2.45);
    \draw[->,line width=.22mm](2.68,-2.25)arc(-190:-335:.15);
    \draw[->,line width=.22mm](3.,-2.222)arc(-195:-325:.29);
    \draw[->,line width=.22mm](3.58,-2.1)arc(-200:-325:.43);
    \draw[->,line width=.22mm](4.55,-1.95)arc(-210:-313:.85);
    \node at(4,-1.3){propagation of $\|\omega\|_{L_x^2}$\\using Prop.\ \ref{uc}};
    \node at(5.5,.75){$|x|=R$};
    \node at(7.5,.8){$(R/T_1^\frac12)^{2A_4}T_1^\frac12$};
    \draw[dashed,line width=.22mm](5.5,-4.3)--(5.5,.5);
    \draw[dashed,line width=.22mm](7.5,-4.3)--(7.5,.5);
    \draw[dashed,line width=.22mm](5.5,-4.3)--(7.5,-4.3);
    \node at(6.5,-3.9){annulus of\\regularity};
    \fill[blue!50!white,opacity=.8](5.6,.43)rectangle(7,.57);
    \draw[line width=.22mm,->](6.3,-1.65)--(6.3,.4);
    \fill[white](5.6,-.92)rectangle(7,-.29);
    \node at(6.3,-.6){backward\\uniqueness};
\end{tikzpicture}
\caption{In the proof of Proposition \ref{prop2} we begin with a vorticity concentration in a parabolic cylinder $Q'$, which in turned is contained in a $Q/2$ where $u$ possesses subcritical bounds. We use Proposition \ref{uc} to propagate the vorticity lower bound into a slice of regularity obtained from Proposition \ref{slices}. Then we iteratively apply Proposition \ref{uc} to locate a vorticity concentration in a distant annulus where  $u$ is regular. In this annulus we may apply a backward uniqueness Carleman inequality to conclude the existence of a $\|u\|_{L_x^d}$ concentration at the final time.
}

\end{figure}

Let us once again fix an arbitrary $T_1\in[A_1^2N_0^{-2},T/100]$. For now, we normalize $z_0=0$ and $T_1=1$. We continue to use the notation of Proposition \ref{prop1} and its proof; in particular let us take $Q$ and $Q'$ satisfying the the conclusion. Let $z':=(t',x')$ be the center of $Q'$. We apply Proposition \ref{slices} centered at $z'+(100(\delta\ell)^2,0)$ (i.e., shifted forward in time) at length scale $R=\delta\ell$. This yields a slice of regularity which, by rotating, we may assume has $\theta=e_1$. Specifically, there is an $I\subset[t'+99(\delta\ell)^2,t'+100(\delta\ell)^2]$ of length $(\delta\ell)^2A_2^{-2}$ such that within
\eqn{
S:=I\times\{x\in\mathbb R^d:\dist(x,x'+\mathbb R_+e_1)\leq 10A_2^{-1}|x_1-x_1'|,\,|x-x'|>20\delta\ell\},
}
we have for $j=0,1,2$
\eq{
\|\grad^ju\|_{L_{t,x}^\infty(S)}\leq A_1^{-1}(\delta\ell/A_2)^{-1-j}.\label{S}
}
Let $t''\in I$ be arbitrary. In order to propagate vorticity concentration into this cone, we apply Proposition \ref{uc} to the function
\eqn{
(t,x)\mapsto\omega(t''-t,x+x'+50\delta\ell e_1)
}
on the interval $[0,C_0(\delta\ell)^2]$ with $t_0=75(\delta\ell)^2$, $t_1=(A_3^{-2}\delta^3\ell)^2$, and $r=\ell/2$. The differential inequality \eqref{differentialinequality} for $\omega$ becomes clear from the coordinate form of the vorticity equation
\eqn{
\dd_t\omega_{ij}-\Delta\omega_{ij}+u\cdot\grad\omega_{ij}+(\dd_iu_k)\omega_{kj}-(\dd_ju_k)\omega_{ik}=0,
}
combined with the estimates in \eqref{finalbounds}. Considering each the terms in the Carleman inequality which takes the form $Z\leq X+Y$, by \eqref{finalconc} the left-hand side obeys
\eqn{
Z\geq A_3^{-O(1)}(\delta\ell)^{d}
}
while for the first term on the right-hand side,
\eqn{
X\leq e^{-1/O(\delta^2)}\ell^{d-4}.
}
The latter is negligible compared to the former given \eqref{ell}; thus the Carleman inequality becomes
\eqn{
\int_{B(x'+50\delta\ell e_1,\ell/2)}|\omega(t'')|^2e^{-|x-x'-50\delta\ell e_1|^2/4t_1}dx\geq\ell^d\exp(-\delta^{-3}).
}
Finally we narrow the domain of integration using the fact that the contribution from outside $B(x'+50\delta\ell e_1,A_3^{-1}\delta\ell)$ is negligible compared to the left-hand side which follows from \eqref{finalbounds} and \eqref{ell}. This yields
\eq{
\int_{B(x'+50\delta\ell e_1,A_3^{-1}\delta\ell)}|\omega(t'')|^2dx\geq\exp(-A_4^5)\label{slicelb}
}
for every $t''\in I$.

Next we apply Proposition \ref{annuli} to find an $R\in[A_4,\exp\exp(A_4)]$ such that
\eq{\label{annulusregularity}
\|\grad^ju\|_{L_{t,x}^\infty([-1,0]\times\{|x|\in[R,R^{2A_4}]\})}\leq A_4^{-1/O(1)}
}
for $j=0,1,2$. Then define $x_*=x'+100Re_1$ and let $\tau=\sup I$. We apply Proposition \ref{iteratedcarleman} to the function
\eqn{
(t,x)\mapsto\omega(t+\tau,x+x')
}
on the interval $[0,4(\delta\ell)^2]$ with $R_0=50\delta\ell$, $\eta=A_2^{-3}$, and $\epsilon=e^{-A_4^6}$ to find
\eq{\label{xstarlowerbound}
\int_{B(x_*,A_3^{-1}R)}|\omega(t)|^2dx\geq e^{-R^{A_4}}
}
for every $t\in[\tau-e^{-3A_4},\tau]$. Note that the initial lower bound follows from \eqref{slicelb} and that we have \eqref{reg} thanks to \eqref{S} and the vorticity equation.

Next we propagate this concentration forward in time using a Carleman inequality for backward uniqueness, see Proposition 4.2 in \cite{tao} (the extension of which to higher dimensions was proved in \cite{p}, Proposition 9). In particular, by applying it to the function $(t,x)\mapsto\omega(-t,x)$ on the interval $[0,1]$ with $r_-=5R$ and $r_+=R^{2A_4}/10$, we have $Z\leq X+Y$ where
\eqn{
Z&\gtrsim e^{-(5R^2)^{A_4}},\\
X&\lesssim e^{-R^{5A_4}}\int_{-1}^0\int_{5R\leq|x|\leq R^{2A_4}/10}e^{2|x|^2/C_0}(|\omega|^2+|\grad\omega|^2)dxdt,\\
Y&\leq e^{
R^{10A_4}}\int_{5R\leq|x|\leq r_+}|\omega(0,x)|^2dx.
}
(Note that this and all subsequent applications of Carleman inequalities are valid because \eqref{differentialinequality} is implied by \eqref{annulusregularity} and the vorticity equation.) Thus there are two cases:
\eq{\label{directcase}
\int_{5R\leq|x|\leq R^{2A_4}/10}|\omega(0,x)|^2dx\geq e^{-R^{20A_4}}
}
and
\eq{\label{indirectcase}
&\int_{-1}^0\int_{5R\leq|x|\leq R^{2A_4}/10}e^{2|x|^2/C_0}(|\omega|^2+|\grad\omega|^2)dxdt\geq e^{R^{A_4}}.
} 
First assuming \eqref{indirectcase}, we essentially follow the proof of Theorem 5.1 in \cite{tao}. By the pigeonhole principle, there exists an $R'\in[5R,R^{2A_4}/10]$ such that
\eqn{
\int_{-1}^0\int_{R'\leq|x|\leq 2R'}(|\omega|^2+|\grad\omega|^2)dxdt\geq e^{-4(R')^2/C_0}.
}
By \eqref{annulusregularity}, the contribution to the left-hand side from the time interval $[-e^{-(R')^2},0]$ is negligible compared to the right so essentially the same lower bound holds with the integral evaluated on $[-1,-e^{-(R')^2}]$. We apply the pigeonhole principle, now in time, to find a $T_0\in[e^{-(R')^2},1]$ in this time interval such that
\eqn{
\int_{-2T_0}^{-T_0}\int_{R'\leq|x|\leq 2R'}(|\omega|^2+|\grad\omega|^2)dxdt\geq e^{-(R')^2}.
}
Having obtained length and time scales where the vorticity concentrates, we cover the annulus $\{R'\leq|x|\leq2R'\}$ by $O(R'/T_0^\frac12)^d$ balls of radius $T_0^\frac12$. The pigeonhole principle then provides an $x_0\in\{R'\leq|x|\leq 2R'\}$ such that
\eqn{
\int_{Q((-T_0,x_0),T_0^{1/2})}(|\omega|^2+|\grad\omega|^2)dxdt\geq e^{-O(R')^2}.
}
Finally we may apply Proposition \ref{uc} on $[0,1000dT_0]$ to the function
\eqn{
(t,x)\mapsto\omega(-t,x+x_0)
}
with $t_0=T_0$, $t_1=C_0^{-3}T_0$, and $r=C_0R'T_0^\frac12$. The Carleman inequality becomes
\eqn{
e^{-O(R')^2}\leq e^{-C_0(R')^2}T_0^{\frac d2}+e^{O(C_0^2(R')^2)}\int_{B(x_0,C_0R'T_0^\frac12)}|\omega(0,x)|^2e^{-C_0^3|x-x_0|^2/4T_0}dx.
}
With a sufficiently large choice of $C_0$, the first term on the right-hand side is negligible compared the the left. Moreover, the contribution to the second term on the right from outside $B(x_0,R'/2)$ is also negligible by \eqref{annulusregularity}. Thus
\eqn{
\int_{B(x_0,R'/2)}|\omega(0,x)|^2dx\geq e^{-C_0^3(R')^2}.
}
In both cases \eqref{directcase} and \eqref{indirectcase}, we can thus conclude
\eqn{
\int_{2R\leq|x|\leq R^{2A_4}/4}|\omega(0,x)|^2dx\geq\exp(-\exp\exp(2A_4)).
}
Now let us fix an $x_*\in\{2R\leq|x|\leq R^{2A_4}/4\}$ where
\eqn{
|\omega(0,x_*)|\geq\exp(-\exp\exp(3A_4)).
}
By repeating the simple mollification argument from \cite{tao} to convert the concentration of vorticity into the critical space, we obtain
\eqn{
\int_{A_4T_1^\frac12\leq|x|\leq\exp\exp(3A_4)T_1^\frac12}|u(0,x)|^ddx\geq\exp(-\exp\exp A_5).
}
At this point we undo the original rescaling so that $T_1$ is explicit. This estimate can be summed over geometrically separated scales $T_1\in[A_1^2N_0^{-2},T/100]$ to conclude
\eqn{
\int_\Rd|u(0,x)|^ddx\geq\exp(-\exp\exp A_5)\log(TN_0^2)
}
which implies the result when compared to the upper bound \eqref{critical}.
\end{proof}

\section{Proof of Theorems \ref{blowupthm} and \ref{regularitythm}}\label{proofoftheorems}

As in \cite{tao}, Theorem \ref{blowupthm} is obtained easily from Theorem \ref{regularitythm} combined with, say, the Prodi-Serrin-Ladyzhenskaya blowup criterion.

\begin{proof}[Proof of Theorem \ref{regularitythm}]
We increase $A$ so that $A\geq C_0$ and rescale so that $t=1$. By Propositions \ref{prop1} and \ref{prop2} in the converse, we have that
\eq{\label{Nstar}
\|P_Nu\|_{L_{t,x}^\infty([\frac12,1]\times\Rd)}\leq A_1^{-1}N
}
for all
\eqn{
N\geq N_*:=2\exp\exp\exp\exp(A_6).
}
Starting with the decomposition $u=\ulin+\unlin$ on $[0,1]$ and differentiating to reach $\omega=\wlin+\wnlin$, we define the enstrophy-type quantities
\eqn{
E_n(t):=\int_\Rd\frac{|\grad^n\wnlin(t)|^2}2dx
}
and compute
\eqn{
E_0'(t)&=-\int_\Rd|\grad\wnlin|^2dx-\int_\Rd\wnlin\cdot\langle\grad\unlin,\wnlin\rangle dx\\
&\quad\quad-\int_\Rd\wnlin\cdot(\langle\grad\unlin,\wlin\rangle+\langle\grad\ulin,\wnlin\rangle+\unlin\cdot\grad\wlin-f)dx\\
&=-X_1+X_2+X_3.
}
Here we have defined $\langle\grad u,\omega\rangle_{ij}:=(\dd_iu_k)\omega_{kj}-(\dd_ju_k)\omega_{ik}$ for a vector field $u$ and 2-form $\omega$ so that we may represent the Lie derivative as $\mathcal L_u\omega=\langle\grad u,\omega\rangle+u\cdot\grad \omega$.

Clearly $X_1\geq0$.
By Littlewood-Paley decomposition and Plancherel we have
\eqn{
X_2(t)&=-\sum_{N_1,N_2,N_3}\int_\Rd P_{N_1}\wnlin\cdot\langle\grad P_{N_2}\unlin,P_{N_3}\wnlin\rangle dx\\
&\lesssim\sum_{N_1\sim N_2\gtrsim N_3}\|P_{N_1}\wnlin\|_{L_x^2(\Rd)}\|P_{N_2}\wnlin\|_{L_x^2(\Rd)}\|P_{N_3}\wnlin\|_{L_x^\infty(\Rd)}.
}
Applying Lemma \ref{frequencylocalized} and \eqref{Nstar} for $N_3$ smaller or larger than $N_*$ respectively, we arrive at
\eqn{
X_2(t)&\lesssim\sum_{N_1}\|P_{N_1}\wnlin(t)\|_{L_x^2(\Rd)}^2(A^{O(1)}N_*^2+A_1^{-1}N_1^2)\\
&\lesssim \|\grad\unlin(t)\|_{L_x^2(\Rd)}^2A^{O(1)}N_*^2+A_1^{-1}X_1.
}
By H\"older's inequality, \eqref{ulin}, \eqref{unlin}, and \eqref{force}, we have for $t\in[\frac12,1]$
\eqn{
X_3(t)\leq(\|\grad\unlin(t)\|_{L_x^2}^2+1)A^{O(1)}.
}
Integrating in time using \eqref{sharpenergy} and Gronwall's inequality, we find that for any $\frac12\leq t_1\leq t_2\leq1$,
\eqn{
E_0(t_2)-E_0(t_1)\leq N_*^2A^{O(1)}.
}
At the same time, by \eqref{sharpenergy}, there exists a $t_0\in[1/2,3/4]$ such that $E_0(t_0)\leq A^{O(1)}$. Thus
\eq{\label{E0}
\sup_{t\in[\frac34,1]}E_0+2\int_\frac34^tE_1(t)dt\leq N_*^2A^{O(1)}.
}
Next we compute using \eqref{sharpequation}
\eqn{
E_n'(t)&=-Y_1+Y_2+Y_3+Y_4+Y_5
}
where
\eqn{
Y_1(t)&=\int_\Rd|\grad^{n+1}\wnlin|^2,\\
Y_2(t)&=-\sum_{k=0}^n\binom nk\int_\Rd\grad^n\wnlin\cdot\langle\grad\grad^{n-k}\unlin,\grad^k\wnlin\rangle dx,\\
Y_3(t)&=-\sum_{k=1}^n\binom nk\int_\Rd\grad^n\wnlin\cdot(\grad^k\unlin\cdot\grad\grad^{n-k}\wnlin)dx,\\
Y_4(t)&=-\sum_{k=1}^n\binom nk\int_\Rd\grad^n\wnlin\cdot(\grad^k\ulin\cdot\grad\grad^{n-k}\wnlin)dx\\
Y_5(t)&=-\int_\Rd\grad^n\wnlin\cdot\grad^n(\langle\grad\unlin,\wlin\rangle+\langle\grad\ulin,\wnlin\rangle-\unlin\cdot\grad\wlin-\curl f)dx.
}
We then take the Littlewood-Paley decompositions and estimate
\eqn{
Y_2(t)&=-\sum_{k=0}^n\binom nk\sum_{N_1,N_2,N_3}\int_\Rd\grad^nP_{N_1}\wnlin\cdot\langle\grad\grad^{n-k}P_{N_3}\unlin,\grad^kP_{N_2}\wnlin\rangle dx\\
&\leq I+II
}
where we decompose based on whether the top order derivatives that fall on the high frequency factors. Specifically, by H\"older, Lemma \ref{bernstein}, \eqref{critical}, and \eqref{Nstar},
\eqn{
I&\lesssim_n\sum_{k=0}^{n}\sum_{N_1\sim N_2\gtrsim N_3}\|\grad^nP_{N_1}\wnlin\|_{L_x^2(\Rd)}\|\grad^kP_{N_2}\wnlin\|_{L_x^2(\Rd)}\|\grad^{n-k+1}P_{N_3}\unlin\|_{L_x^\infty(\Rd)}\\
&\lesssim\sum_{k=0}^{n}\sum_{N_1\sim N_2}\|\grad^nP_{N_1}\wnlin\|_{L_x^2(\Rd)}\|\grad^kP_{N_2}\wnlin\|_{L_x^2(\Rd)}(A^{O(1)}N_*^{n-k+2}+A_1^{-1}N_1^{n-k+2})\\
&\lesssim \sum_{k=0}^{n}A^{O(1)}N_*^{n-k+2}E_k(t)^\frac12E_n(t)^\frac12+A_1^{-1}Y_1(t),
}
and
\eqn{
II&\lesssim_n\sum_{k=1}^{n-1}\sum_{N_1\sim N_2\gtrsim N_3}\|\grad^kP_{N_1}\wnlin\|_{L_x^2(\Rd)}\|\grad^{n-k}P_{N_2}\wnlin\|_{L_x^2(\Rd)}\|\grad^nP_{N_3}\wnlin\|_{L_x^\infty(\Rd)}\\
&\lesssim\sum_{k=1}^{n-1}\sum_{N_1\sim N_2}\|\grad^kP_{N_1}\wnlin\|_{L_x^2(\Rd)}\|\grad^{n-k}P_{N_2}\wnlin\|_{L_x^2(\Rd)}(A^{O(1)}N_*^{n+2}+A_1^{-1}N_1^{n+2})\\
&\lesssim\sum_{k=1}^{n-1}A^{O(1)}N_*^{n+2}E_k(t)^\frac12E_{n-k}(t)^\frac12+A_1^{-1}Y_1(t).
}
Next, $Y_3(t)$ contains essentially the same terms and admits the same bounds (note crucially the exclusion of $k=0$ by incompressiblity). By Cauchy-Schwarz and \eqref{ulin},
\eqn{
Y_4(t)\lesssim_n\sum_{k=1}^nA^{O(1)}E_n(t)^\frac12E_{n-k+1}(t)^\frac12.
}
Finally, by \eqref{ulin}, \eqref{unlin}, \eqref{force}, and integration by parts,
\eqn{
Y_5(t)\lesssim_n A^{O(1)}E_n(t)^\frac12\Big(1+\sum_{k=0}^nE_k(t)^\frac12\Big)+A^{O(1)}.
}
In total, combining some terms with Young's inequality,
\eqn{
E_n'(t)\leq A^{O_n(1)}N_*^2E_n(t)+N_*^{2n+2}\sum_{k=0}^{n-1}E_k(t)+A^{O_n(1)}.
}
Inductively applying Gronwall's inequality (at each step using the pigeonhole principle to find an initial time), starting with \eqref{E0} as a base case, implies
\eqn{
\sup_{t\in[t_n,1]}\int_\Rd|\grad^n\wnlin(t)|^2dx+\int_{t_n}^1\int_\Rd|\grad^{n+1}\wnlin(t)|^2dxdt\leq N_*^{O_n(1)}
}
for an increasing sequence $t_n\in[\frac12,1]$. The claimed $L_{t,x}^\infty$ estimates are immediate by \eqref{ulin} and Sobolev embedding, taking $n$ sufficiently large depending on $d$.
\end{proof}

\bibliographystyle{abbrv}
\bibliography{references}

\begin{thebibliography}{10}

\bibitem{albritton}
D.~Albritton.
\newblock Blow-up criteria for the {Navier-Stokes} equations in non-endpoint
  critical {Besov} spaces.
\newblock {\em Analysis \& PDE}, 11(6):1415--1456, 2018.

\bibitem{ab}
D.~Albritton and T.~Barker.
\newblock Global weak {B}esov solutions of the {Navier--Stokes} equations and
  applications.
\newblock {\em Archive for Rational Mechanics and Analysis}, 232(1):197--263,
  2019.

\bibitem{barkerprange}
T.~Barker and C.~Prange.
\newblock Quantitative regularity for the {N}avier--{S}tokes equations via
  spatial concentration.
\newblock {\em Communications in Mathematical Physics}, pages 1--76, 2021.

\bibitem{calderon}
C.~P. Calder{\'o}n.
\newblock Existence of weak solutions for the {N}avier-{S}tokes equations with
  initial data in {$L^p$}.
\newblock {\em Transactions of the American Mathematical Society},
  318(1):179--200, 1990.

\bibitem{cp}
J.-Y. Chemin and F.~Planchon.
\newblock Self-improving bounds for the {Navier-Stokes} equations.
\newblock {\em Bulletin de la Soci{\'e}t{\'e} Math{\'e}matique de France},
  140(4):583--597, 2012.

\bibitem{dongdu}
H.~Dong and D.~Du.
\newblock The {N}avier-{S}tokes equations in the critical {L}ebesgue space.
\newblock {\em Communications in Mathematical Physics}, 292(3):811--827, 2009.

\bibitem{dw}
H.~Dong and K.~Wang.
\newblock Interior and boundary regularity for the navier-stokes equations in
  the critical lebesgue spaces.
\newblock {\em arXiv preprint arXiv:1809.06712}, 2018.

\bibitem{bu}
L.~Escauriaza, G.~Seregin, and V.~{\v{S}}ver{\'a}k.
\newblock Backward uniqueness for parabolic equations.
\newblock {\em Archive for rational mechanics and analysis}, 169(2):147--157,
  2003.

\bibitem{ess}
L.~Escauriaza, G.~A. Seregin, and V.~Sverak.
\newblock On ${L}_{3,\infty}$-solutions of the {N}avier-{S}tokes equations and
  backward uniqueness.
\newblock {\em Russian Mathematical Surveys}, 58(2):211--250, 2003.

\bibitem{gkp}
I.~Gallagher, G.~S. Koch, and F.~Planchon.
\newblock Blow-up of critical {B}esov norms at a potential {Navier--Stokes}
  singularity.
\newblock {\em Communications in Mathematical Physics}, 343(1):39--82, 2016.

\bibitem{lady}
O.~A. Ladyzhenskaya.
\newblock On the uniqueness and on the smoothness of weak solutions of the
  {Navier--Stokes} equations.
\newblock {\em Zapiski Nauchnykh Seminarov POMI}, 5:169--185, 1967.

\bibitem{leray}
J.~Leray.
\newblock Sur le mouvement d'un liquide visqueux emplissant l'espace.
\newblock {\em Acta mathematica}, 63:193--248, 1934.

\bibitem{p}
S.~Palasek.
\newblock Improved quantitative regularity for the {N}avier--{S}tokes equations
  in a scale of critical spaces.
\newblock {\em Archive for rational mechanics and analysis}, 242(3):1479--1531,
  2021.

\bibitem{phuc}
N.~C. Phuc.
\newblock The {Navier--Stokes} equations in nonendpoint borderline {Lorentz}
  spaces.
\newblock {\em Journal of Mathematical Fluid Mechanics}, 17(4):741--760, 2015.

\bibitem{prodi}
G.~Prodi.
\newblock Un teorema di unicita per le equazioni di {Navier-Stokes}.
\newblock {\em Annali di Matematica pura ed applicata}, 48(1):173--182, 1959.

\bibitem{seregin}
G.~Seregin.
\newblock A certain necessary condition of potential blow up for
  {Navier-Stokes} equations.
\newblock {\em Communications in Mathematical Physics}, 3(312):833--845, 2012.

\bibitem{serrin}
J.~Serrin.
\newblock {\em On the interior regularity of weak solutions of the
  {Navier-Stokes} equations}.
\newblock Mathematics Division, Air Force Office of Scientific Research, 1961.

\bibitem{othertao}
T.~Tao.
\newblock Localisation and compactness properties of the {N}avier--{S}tokes
  global regularity problem.
\newblock {\em Analysis \& PDE}, 6(1):25--107, 2013.

\bibitem{tao}
T.~Tao.
\newblock Quantitative bounds for critically bounded solutions to the
  {N}avier-{S}tokes equations.
\newblock In A.~Kechris, N.~Makarov, D.~Ramakrishnan, and X.~Zhu, editors, {\em
  Nine Mathematical Challenges: An Elucidation}, volume 104. American
  Mathematical Society, 2021.

\end{thebibliography}

\end{document}